\documentclass[12pt]{article}
\usepackage[usenames,dvipsnames]{color}
\usepackage{xr,graphicx,amsthm,amsmath,amssymb,xargs,srcltx,bbm,todonotes,stmaryrd,framed,a4wide,accents,sectsty}
\usepackage[shortlabels]{enumitem}
\usepackage{framed}
\usepackage{pdfsync}
\newcommandx\tedclusterr[2][1=n,2=\dhinterseq]{{\widetilde{\boldsymbol\xi}}^*_{#1,#2}} 
\newcommand\TEDclusterrandomr{\widehat{\boldsymbol\xi}^*}

\usepackage{booktabs}
\newcommand{\ra}[1]{\renewcommand{\arraystretch}{#1}}

\usepackage[colorlinks]{hyperref}

\newcommand{\remarkend}{\hfill $\oplus$}

\newcommandx{\conditiondhsumpaper}[2]{\ensuremath{\mathcal{S}(#1,#2)}} 
\newcommandx\tedclustersl[2][1=n,2=\dhinterseq]{{\widetilde{\boldsymbol\mu}}^*_{#1,#2}} 

\newcommand\TEDclusterrandomsl{\widehat{\boldsymbol\mu}^*} 

\newcommand\statinterseqn{k_n}

\subsubsectionfont{\itshape}

\usepackage{cleveref}

\newtheorem{theorem}{Theorem}[section]
\newtheorem{proposition}[theorem]{Proposition}
\newtheorem{lemma}[theorem]{Lemma}

\newtheorem{definition}[theorem]{Definition}
\newtheorem{hypothesis}[theorem]{Assumption}

\theoremstyle{remark}

\newtheorem{remark}[theorem]{Remark}

\crefname{hypothesis}{assumption}{assumptions}

\crefname{lemma}{Lemma}{Lemmas}
\crefname{theorem}{Theorem}{Theorems}
\crefname{proposition}{Proposition}{Propositions}
\crefname{exercise}{Exercise}{Exercises}
\crefname{corollary}{Corollary}{Corollaries}

\numberwithin{equation}{section}

\newcommandx\TEP[1][1=]{\mathbb{G}^{#1}} 

\newcommand\tepcluster{{\mathbb{G}}}
\newcommandx\tedcluster[2][1=n,2=\dhinterseq]{{\widetilde{\boldsymbol\nu}}^*_{#1,#2}} 
\newcommandx\tedclusterindep[2][1=n,2=\dhinterseq]{{\widetilde{\boldsymbol\nu}}^\indep_{#1,#2}} 
\newcommandx\tedclustermdep[2][1=n,2=\dhinterseq]{{\widetilde{\boldsymbol\nu}}^{*(m)}_{#1,#2}} 

\newcommand\TEDclusterrandom{\widehat{\boldsymbol\nu}^*} 

\newcommandx\tedclusterrandom[2][1=n,2=\dhinterseq]{{\widehat{\boldsymbol\nu}}^*_{#1,#2}}

\newcommand{\ANSJB}{{\rm ANSJB}}

\newcommand\constant{\mathrm{cst}}

\newcommandx\envelope[1][1={\mathbf H}]{{\mathbf {#1}}}

\newcommand\canditheta{{\vartheta}}
\newcommandx\stoploss[1][1={\rm stoploss}]{\theta_{#1}}
\newcommandx\stoplossest[1][1={{\rm stoploss},n}]{\widetilde\theta_{#1}}
\newcommandx\stoplossestk[1][1={{\rm stoploss},n,\statinterseq}]{\widehat\theta_{#1}}
\newcommandx\largedeviation[1][1={\rm largedev}]{\theta_{#1}}
\newcommandx\largedeviationest[1][1={{\rm largedev},n}]{\widetilde\theta_{#1}}
\newcommandx\largedeviationestk[1][1={{\rm largedev},n,\statinterseq}]{\widehat\theta_{#1}}
\newcommandx\ruin[1][1={\rm ruin}]{\theta_{#1}}
\newcommandx\ruinest[1][1={{\rm ruin},n}]{\widetilde\theta_{#1}}
\newcommandx\ruinestk[1][1={{\rm ruin},n,\statinterseq}]{\widehat\theta_{#1}}

\newcommandx\clfunc[1][1=h]{#1} 

\newcommandx\anticl[2][2=\dhinterseq]{\conditionS(#1,#2)}

\newcommandx\anticlpsi[3][1=\tepseq,2=\dhinterseq,3=\psi]{\conditionS(#1,#2,#3)}

\newcommandx{\norm}[2][1=]{\left|#2\right|_{#1}} 
\newcommandx{\lpnorm}[3][1=,3=]{\|#2\|_{#1}^{#3}}

\newcommandx{\matrixnorm}[2][2=]{\left\|#1\right\|_{#2}} 
\newcommandx{\matrixnormseries}[3][2=,3=]{\left\|#1\right\|_{#2,#3}} 
\newcommandx{\linftynorm}[2][1=]{\|#2\|_{#1}} 
\newcommandx{\anynorm}[2][2=]{\left|#1\right|_{#2}} 

\newcommandx{\supnormclass}[3][3=]{\left\|#1\right\|_{#2}^{#3}} 

\newcommandx{\sphere}[2][1=]{\mathbb{S}_{#1}^{#2}}
\newcommandx{\oball}[3][3=]{B_{#3}(#1,#2)} 
\newcommandx{\cball}[3][3=]{\overline{B}_{#3}(#1,#2)} 
\newcommandx{\coball}[3][3=]{B_{#3}^c(#1,#2)} 
\newcommandx{\ccball}[3][3=]{\overline{B}_{#3}^c(#1,#2)} 
\newcommandx\cone[1][1=C]{\mathcal{#1}}

\newcommandx\conej[1][1=\mathbf{j}]{\mathcal{C}_{#1}}   

\newcommandx{\coneindex}[2][2=\mathcal{C}]{#1_#2}

\newcommand\Nset{\mathbb{N}}
\newcommand\Zset{\mathbb{Z}}

\newcommand\Rset{\mathbb{R}}

\newcommandx\borel[1][1=\csms]{\mathcal{B}(#1)}   

\newcommand{\bszero}{{\boldsymbol0}}

\newcommand\indep{\dag}  

\newcommand\bsQ{\boldsymbol{Q}}

\newcommand\bsx{\boldsymbol{x}}
\newcommand\bsX{\boldsymbol{X}}

\newcommand\bsy{\boldsymbol{y}}
\newcommand\bsY{\boldsymbol{Y}}

\newcommand\bsTheta{\boldsymbol{\Theta}}

\newcommand\bsnu{\boldsymbol{\nu}}

\newcommand\tailmeasure{\bsnu} 
\newcommand\tailmeasurestar{{\bsnu}^*} 
\newcommand\ltwotmsmetric{\boldsymbol{\rho}^*} 

\newcommand\exc{\mce}  
\newcommand\anchor{\mca} 

\newcommandx{\numult}[2][2=]{\bsnu_{\boldsymbol{#1}_{#2}}} 

\newcommand\backest\Upsilon

\newcommandx{\nualphap}[2][1=\alpha,2=p]{\nu_{#1,#2}}  

\newcommandx\numultcondi[2][2=]{\boldsymbol{\nu}_{\boldsymbol{#1}_{#2}}}

\newcommandx\chain[1][1=Y]{\mathbb{#1}}

\newcommandx{\scalingseq}[1][1=n]{c_{#1}}  
\newcommandx{\scalingfunction}[1][1=]{c#1} 

 \newcommandx{\absquantileseq}[2][1=]{a_{#1#2}}

\newcommandx{\dhinterseq}[1][1=n]{r_{#1}}  
\newcommandx{\dhinterseqsmall}[1][1=n]{\ell_{#1}}  
\newcommandx{\dinterseq}[1][1=n]{r_{#1}}  
\newcommandx{\tepseq}[1][1=n]{u_{#1}}  

\newcommandx{\interseq}[1][1=n]{k_{#1}}  

\newcommandx{\scalingseqcone}[1][1=n]{c_{#1}} 

\newcommandx{\scalingseqhidden}[1][1=n]{\tilde{c}_{#1}}
\newcommandx{\scalingfunctionhidden}[1][1=]{\tilde{c}{#1}}
\newcommandx{\scalingseqcev}[1][1=n]{c^*_{#1}}

\newcommand\conditionS{\ensuremath{\mathcal{S}}} 
\newcommandx{\conditiondh}[2][1=\dhinterseq,2=\scalingseq]{\ensuremath{\mathcal{A}\mathcal{C}(#1,#2)}} 
\newcommandx{\conditionANSJB}[1][1=\scalingseq]{\mathrm{ANSJB}(\dhinterseq,#1)}
\newcommandx{\conditiondhsum}[1][1=\scalingseq]{\ensuremath{\mathcal{S}(\dhinterseq,#1)}} 
\newcommandx{\conditiondhsumW}[1][1=\scalingseq]{\ensuremath{\mathcal{SW}(\dhinterseq,#1)}} 

\newcommand\conditionR{\ensuremath{\mathcal{R}}} 

\newcommand\convdistr{\stackrel{\mbox{\tiny\rm d}}{\longrightarrow}} 

\newcommand\convprob{\stackrel{\tiny \mathbb{P}}{\longrightarrow}}

\newcommand\vaguelysharp{vaguely$^\#$}

\newcommandx\prohodistsym[1][1=]{\rho_{#1}}
\newcommandx\prohodist[3][3=]{\rho_{#3}(#1,#2)}

\newcommand\rmd{\mathrm{d}} 

\newcommand\esp{\mathbb E}
\newcommand\pr{\mathbb P}
\newcommand\var{\mathrm{var}}
\newcommand\cov{\mathrm{cov}}

\newcommandx{\autocov}[1][1=]{\gamma_{#1}}

\newcommandx{\cdfnorm}[1][1=\bsX]{H_{#1}}
\newcommandx{\tailcdfnorm}[1][1=\bsX]{\overline{H}_{#1}}

\newcommand\ind[1]{\mathbbm{1}{\left\{#1\right\}}}

\newcommand\mca{\mathcal A}
\newcommand\mcb{\mathcal B}

\newcommand\mce{\mathcal E}
\newcommand\mcf{\mathcal F}
\newcommand\mcg{\mathcal G}

\newcommand\mch{\mathcal H}

\newcommand\mci{\mathcal I}

\newcommand\mck{\mathcal K}
\newcommand\mcl{\mathcal L}

\newcommand\bbR{\mathbb{R}}

\newcommand\bbZ{\mathbb{Z}}

\newcommandx\test[2][1=X]{{#1}_{#2}}
\newcommandx\orderstat[3][1=X]{{#1}_{(#2:#3)}}
\newcommand\statinterseq{k}

\newcommandx{\sequence}[3][2=\Zset,3=j]{\{#1_{#3},#3\in#2\}}
\newcommandx{\sequenceshort}[2][2=j]{\{#1_#2\}}
\newcommandx\sequ[3][2=j,3=\mathbb{Z}]{\{#1_#2,#2\in#3\}}
\newcommandx\sequnorm[3][3=j,2=\mathbb{Z}]{\{\norm{#1_#3},#3\in#2\}}
\newcommandx\sequnormq[4][2=,4=j,3=\mathbb{Z}]{\{\norm{#1_#4}^{#2},#4\in#3\}}

\newcommandx\uncompactd[2][1=d]{(\overline{\Rset}^{#1})^{#2}\setminus\{\boldsymbol0\}}
\newcommandx{\barrsetproduct}[2][1=d]{(\overline{\Rset}^{#1})^{#2}}
\newcommandx{\rsetproduct}[2][1=d]{(\Rset^{#1})^{#2}}

\newcommand{\metricspace}{\csms}

\newcommand\spaceD{\mathbb{D}}  

\newcommandx\csms[1][1=E]{\mathsf{#1}}   
\newcommandx\borelcsms[1][1=E]{\mathcal{#1}}   
\newcommandx\mplusx[1][1=]{\mathcal{M}#1}  
\newcommandx\mplusxp[1][1=]{\mathcal{N}{#1}} 
\newcommandx\mplusxpb[1][1=\borelcsms]{\mathcal{N}_{pb}({#1})}  
\newcommandx\mplusxpone[1][1=\borelcsms]{\mathcal{N}_{p1}({#1})}  
\newcommandx\mplusxpeps[1][1=\borelcsms]{\mathcal{N}_{p\epsilon}({#1})}  
\newcommandx\mplusxf[1][1=\borelcsms]{\mathcal{M}_f({#1})}
\newcommandx\mplusxpf[1][1=\borelcsms]{\mathcal{N}_{pf}({#1})} 
\newcommandx\mplusxps[1][1=\borelcsms]{\mathcal{N}_{ps}({#1})} 
\newcommandx\mplusxpS[1][1=\borelcsms]{\mathcal{N}_{pS}({#1})} 
\newcommandx\mplusxpsc[1][1=\borelcsms]{\mathcal{N}_{psc}({#1})} 

\newcommand\distance{\mathrm{d}}  
\newcommandx\metric[1][1=\metricspace]{\distance_{#1}} 
\newcommandx\metricmcg{\rho} 

\newcommandx\bracknum[3][2=\mch]{N_{[\,]}(#1,#2,#3)} 
\newcommandx\bracknumarray[2][2=\mch]{N_{[\,]}(#1,#2,L^2_n)} 
\newcommandx\entropynum[3][3=\mch]{N(#1,#3,#2)} 

\newcommandx\process[1][1=X]{\mathbb{#1}}

\newcommandx\hillest[3][1=n,3=]{\widehat{\gamma}_{#1,#2}^{#3}}
\newcommandx\hillmoment[2][1=n]{\widehat{\gamma}_{#1,#2}^{(M)}}

\newcommand\lzero{\ell_0}
\newcommand\tildelzero{\tilde{\ell}_0}

\newcommand\lone{\ell_1}

\newcommandx\lalpha[1][1=\alpha]{\ell_{#1}}

\newcommand\shift{B}
\newcommand\backshift{B}

\newcommand\iid{i.i.d.}

\newcommand\wrt{with respect to}
\newcommand\ie{i.e.}

\newcommand\semimetric{semi-metric}
\newcommand\rhs{right-hand side}

\newcommand\realvalued{real-valued}

\newcommand\nonnegative{non-negative}

\newcommand{\nonzero}{non zero}

\newcommand\shiftinvariant{shift-invariant}

\newcommand\Qseq{conditional spectral tail process}

\begin{document}
\title{Estimation of cluster functionals for regularly varying time series: runs estimators}

\author{Youssouph Cissokho\thanks{University of Ottawa}\and Rafa{\l} Kulik\thanks{University of Ottawa} }

\date{\today}
\maketitle

\begin{abstract}
Cluster indices describe extremal behaviour of stationary time series.
We consider runs estimators of cluster indices. Using a modern theory of multivariate, regularly varying time series, we obtain central limit theorems under conditions that can be easily verified for a large class of models. In particular, we show that blocks and runs estimators have the same limiting variance.
\end{abstract}

\section{Introduction}
Consider
a stationary, regularly varying $\Rset^d$-valued time series $\bsX=\sequence{\bsX}$.
We are interested in its extremal behaviour. A classical approach to this problem is to calculate the \textit{extremal index}. If $|\cdot|$ is an arbitrary norm on $\Rset^d$, then the extremal index $\theta$ (if exists) of $\{|\bsX_j|,j\in\Zset\}$ is defined as a parameter in the limiting distribution of the maxima. With $Q$ being the quantile function of $|\bsX_0|$ and $a_n=Q(1-1/n)$ we have
\begin{align*}
\lim_{n\to\infty}\pr(a_n^{-1}\max_{j=1,\ldots,n}\{|\bsX_1|,\ldots,|\bsX_n|\}\leq x)=\exp(-\theta x^{-\alpha})\;, \ \ x>0\;.
\end{align*}
The parameter $\theta\in (0,1]$ indicates the amount of clustering, with $\theta=1$ (the case of extremal independence) meaning no-clustering of large values.

The extremal index is just one parameter that describes clustering of extremes.
Informally speaking, it arises as the limit
$$
\lim_{n\to\infty}\frac{\esp[H((\bsX_{1},\ldots,\bsX_{\dhinterseq})/\tepseq)]}{\dhinterseq \pr(\norm{\bsX_0}>\tepseq)}\;,
$$
for the particular choice of function $H:(\Rset^d)^{\Zset}\to\Rset$, and a suitable choice of the scaling sequence $\tepseq\to\infty$ and the block size $\dhinterseq\to\infty$.  (Formally speaking, $(\bsX_{1},\ldots,\bsX_{\dhinterseq})$ is a random element of $(\Rset^d)^{\dhinterseq}$, while the domain of $H$ is $(\Rset^d)^{\Zset}$. This inconsistency will be explained later).

In particular, the extremal index is achieved by applying a suitable functional to a cluster:
\begin{align*}
H((\bsX_{1},\ldots,\bsX_{\dhinterseq})/\tepseq)=\ind{
	\max\{\norm{\bsX_{1}},\ldots,\norm{\bsX_{\dhinterseq}}\}>\tepseq}\;.
\end{align*}
That is,
\begin{align}\label{eq:extremal-index-1}
\theta=\lim_{n\to\infty}\frac{\pr(\max\{\norm{\bsX_{1}},\ldots,\norm{\bsX_{\dhinterseq}}\}>\tepseq)}{\dhinterseq \pr(\norm{\bsX_0}>\tepseq)}\;.
\end{align}
Informally speaking, a cluster is a triangular array $(\bsX_1/\tepseq,\ldots,\bsX_{\dhinterseq}/\tepseq)$ with $\dhinterseq,\tepseq\to\infty$ that converges in distribution in a certain sense. Cluster indices are obtained by applying the appropriate  functional $H$ to the cluster.
The functionals are defined on $(\Rset^d)^\Zset$, {the space of $\Rset^d$-valued sequences,} and are such that their values do not depend on coordinates that are equal to zero.
More precisely, for $\bsX=\{\bsX_j,j\in\Zset\}\in (\Rset^d)^\Zset$ and $i\leq j\in\Zset$,  we denote $\bsX_{i,j}=(\bsX_i,\ldots, \bsX_j)\in (\Rset^d)^{(j-i+1)}$.
Then, we identify $H(\bsX_{i,j})$ with
$H((\bszero,\bsX_{i,j},\bszero))$, where $\bszero\in (\Rset^d)^\Zset$ is the zero sequence. Such functionals $H$ will be called \textit{cluster functionals}.

Let $\norm{\cdot}$ be an arbitrary norm on $\Rset^d$ and $\{\tepseq\}$, $\{\dhinterseq\}$ be such that
\begin{align}\label{eq:rnbarFun0}
\lim_{n\to\infty}\tepseq=\lim_{n\to\infty}\dhinterseq =\lim_{n\to\infty}n\pr(\norm{\bsX_0}>\tepseq) = \infty\;, \nonumber\\
\lim_{n\to\infty}\dhinterseq/n=\lim_{n\to\infty}\dhinterseq\pr(\norm{\bsX_0}>\tepseq) = 0\;.
\tag{$\conditionR(\dhinterseq,\tepseq)$}
\end{align}
Given a cluster functional $H$ on $(\Rset^d)^\Zset$, we want to estimate the limiting
quantity
\begin{align}
\tailmeasurestar(H)=  \lim_{n\to\infty}  \tailmeasurestar_{n,\dhinterseq}(H)=\lim_{n\to\infty} \frac{\esp[H(\bsX_{1,\dhinterseq}/\tepseq)]}{\dhinterseq \pr(\norm{\bsX_0}>\tepseq)}\;.  \label{eq:thelimitwhichisnolongercalledbH}
\end{align}
To guarantee existence of the limit we will require additional anticlustering assumptions on the time series $\sequence{\bsX}$. For $\bsx=\{\bsx_j,j\in\Zset\}\in(\Rset^d)^\Zset$ define $\bsx^*=\sup_{j\in\Zset}|\bsx_j|$. The cluster indices of interest are, among others:
\begin{itemize}
	\item the extremal index obtained with $H_1(\bsx)=\ind{\bsx^*>1}$, $\bsx=\{\bsx_j,j\in\Zset\}\in(\Rset^d)^\Zset$;
	\item the cluster size distribution obtained with
	\begin{align}\label{eq:h2}
	H_2(\bsx)=\ind{\sum_{j\in\Zset}\ind{\norm{\bsx_j}>1}=m}\;, \ \   \bsx=\{\bsx_j,j\in\Zset\}\in(\Rset^d)^\Zset\;, \ \  m\in\Nset \;;
	\end{align}
	\item the stop-loss index of a univariate time series obtained with
	\begin{align}\label{eq:h3}
	H_3(\bsx) = \ind{\sum_{j\in\Zset} (x_j-1)_+ > \eta}\;, \ \    \bsx=\{\bsx_j,j\in\Zset\}\in\Rset^\Zset\;, \ \    \eta>0\;;
	\end{align}
	\item the large deviation index of a univariate time series obtained with
	\begin{align}\label{eq:h4}
	H_4(\bsx)=\ind{K(\bsx)>1}\;, \ \  K(\bsx) = \left(\sum_{j\in\Zset} x_j\right)_+\;, \ \
	\bsx=\{\bsx_j,j\in\Zset\}\in\Rset^\Zset\;;
	\end{align}
	\item the ruin index of a univariate time series obtained with
	\begin{align}\label{eq:h5}
	H_5(\bsx)=\ind{K(\bsx)>1}\;, \ \   K(\bsx) = \sup_{i\in\Zset} \left(\sum_{j\leq i}
	x_j\right)_+\;, \ \ \bsx=\{\bsx_j,j\in\Zset\}\in\Rset^\Zset\;.
	\end{align}
\end{itemize}
As indicated above, the extremal index is the classical quantity that arises in the extreme value theory for dependent sequences. Similarly, the cluster size distribution has been studied in \cite{hsing:1991} and \cite{drees:rootzen:2010}.
The large deviation index was studied under the name \textit{cluster index} in \cite{mikosch:wintenberger:2013,mikosch:wintenberger:2014}. It quantifies the effect of dependence in large deviations results.

Several methods of estimation of the limit $\tailmeasurestar(H)$ in \eqref{eq:thelimitwhichisnolongercalledbH} may be
employed. The natural one is to consider a
statistics based on disjoint blocks of size $\dhinterseq$, cf. \cite{drees:rootzen:2010} and \cite{kulik:soulier:2020},
\begin{align*}
\tedcluster(H):= \frac{1}{n\pr(\norm{\bsX_0}>\tepseq)}   \sum_{i=1}^{m_n}
H(\bsX_{(i-1)\dhinterseq+1,i\dhinterseq}/\tepseq) \;,
\end{align*}
where $m_n=[n/\dhinterseq]$ is the number of disjoint blocks.
The data-based estimator is constructed as follows. Let $\statinterseqn\to\infty$ be a sequence of integers and define $\tepseq$ by   $\statinterseqn=n\pr(\norm{\bsX_0}>\tepseq)$. Let $\orderstat[\norm{\bsX}]{n}{1}\leq \cdots \leq \orderstat[\norm{\bsX}]{n}{n}$ be order statistics from $\norm{\bsX_1},\ldots,\norm{\bsX_n}$.
Define
\begin{align}\label{eq:feasible-estimator-of-cluster-index-consistency}
\TEDclusterrandom_{n,\dhinterseq}(H)
&:= \frac{1}{\statinterseqn} \sum_{i=1}^{m_n} H(\bsX_{(i-1)\dhinterseq+1,i\dhinterseq}/\orderstat[\norm{\bsX}]{n}{n-\statinterseqn}) \; .
\end{align}
The general asymptotic theory for disjoint blocks estimators was developed in \cite{drees:rootzen:2010}. See also \cite[Chapter 10]{kulik:soulier:2020}. The limiting variance of the disjoint blocks estimator can be represented as
\begin{align}\label{eq:limiting-variance-disjoint}
\tailmeasurestar& (\{H-\tailmeasurestar(H)\exc\}^2)\;,
\end{align}
where $\exc(\bsx)=\sum_{j\in\Zset}\ind{\norm{\bsx_j}>1}$. This result was established (implicitly) in \cite{drees:rootzen:2010}, but the form of the limiting variance is again given in \cite[Chapter 10]{kulik:soulier:2020}.

Another approach to estimation of $\tailmeasurestar(H)$ is to consider the sliding blocks statistics
\begin{align}\label{eq:sliding-block-estimator-nonfeasible-1}
\tedclustersl(H):=\frac{1}{q_n \dhinterseq \pr(\norm{\bsX_0}>\tepseq)}\sum_{i=0}^{q_n-1}H\left(\bsX_{i+1,i+\dhinterseq}/\tepseq\right)\;
\end{align}
and
and the corresponding estimator defined in terms of order statistics:
\begin{align}\label{eq:sliding-block-estimator-feasible}
\TEDclusterrandomsl_{n,\dhinterseq}(H)=\frac{1}{\dhinterseq\statinterseqn}
\sum_{i=0}^{q_n-1}H\left(\bsX_{i+1,i+\dhinterseq}/\orderstat[\norm{\bsX}]{n}{n-\statinterseqn}\right)\;.
\end{align}
Here, $q_n=n-\dhinterseq-1$ is the number of sliding blocks.
In \cite{drees:neblung:2020} the authors used the framework of \cite{drees:rootzen:2010} and showed that the limiting variance of the sliding blocks estimator never exceeds that of the disjoint blocks estimator. In case of the extremal index, both variances were proven to be equal. In \cite{cissokho:kulik:2021} it was shown that the limiting variances for both disjoint and sliding blocks estimators agree and are given by the expression in \eqref{eq:limiting-variance-disjoint} for an arbitrary choice of $H$. We note at this point that the methodology used in \cite{drees:rootzen:2010,drees:neblung:2020,kulik:soulier:2020,cissokho:kulik:2021} fits into Peak Over Threshold (PoT) framework. On the other hand, in the Block Maxima (BM) framework, sliding blocks estimators yield typically smaller variance; see \cite{bucher:segers:2018disjoint,bucher:segers:2018sliding}. As of this moment, there is no thorough explanation of these phenomena and no formal comparison between PoT and BM framework. See \cite{ferreira:dehaan:2015} for some partial results and \cite{bucher:zhou:2018} for a recent review.

In the present paper we are interested in the so-called \textit{runs estimators}. In the context of the extremal index, this approach goes back to \cite{weissman:novak:1998} and stems from the following representation of the extremal index:
\begin{align}\label{eq:extremal-index-2}
\theta=\lim_{n\to\infty}\pr(\max\{\norm{\bsX_{1}},\ldots,\norm{\bsX_{\dhinterseq}}\}\leq\tepseq\mid \norm{\bsX_0}>\tepseq)\;.
\end{align}
We note that
\begin{align*}
\theta=\lim_{n\to\infty}\frac{1}{\pr(\norm{\bsX_0}>\tepseq)}
\esp[\ind{\anchor(({\bsX_{0}},\ldots,\bsX_{\dhinterseq})/\tepseq)=0}\ind{\norm{\bsX_0}>\tepseq}]\;,
\end{align*}
where
\begin{align*}
\anchor(\bsx)=
\sup\{j:\norm{\bsx_j}>1\}\;
\end{align*}
gives the position of the last exceedence above 1 in a particular block.
Recall again the convention
$\anchor(({\bsX_{0}},\ldots,\bsX_{\dhinterseq})/\tepseq)
=\anchor((\bszero,{\bsX_{0}},\ldots,\bsX_{\dhinterseq})/\tepseq)$. Then, $\anchor$ is an example of so-called \textit{anchoring map}. Special cases of anchoring maps were considered in  \cite{hashorva:2018} and \cite{basrak:planinic:2018}, while in \cite{kulik:soulier:2020} their connection to cluster indices $\tailmeasurestar(H)$ was thoroughly investigated.
It turns out that with an arbitrary choice of the anchoring map $\anchor$ we have
\begin{align*}
\tailmeasurestar(H)=\esp[H^{\anchor}(\bsY)]\;,
\end{align*}
where
\begin{align*}
H^{\anchor}(\bsx)=H(\bsx)\ind{\anchor(\bsx)=0}\ind{\norm{\bsx_0}>1}\;.
\end{align*}
This motivates the following runs statistics:
\begin{align}
&\tedclusterr(H^{\anchor})=\frac{1}{n\pr(\norm{\bsX_0}>\tepseq)}
\sum_{i=\dhinterseq+1}^{n-\dhinterseq}H^{\anchor}\left(\bsX_{i-\dhinterseq,i+\dhinterseq}/\tepseq\right)\label{eq:process-xii}\;.
\end{align}
Indeed,
under the appropriate conditions,
\Cref{lem:tailprocesstozero} gives
\begin{align*}
&\lim_{n\to\infty}\esp[\tedclusterr(H^{\anchor})]
=\esp[H^{\anchor}(\bsY)]=\tailmeasurestar(H)\;.
\end{align*}
The data-based
runs estimator is then
\begin{align*}
&\TEDclusterrandomr_{n,\dhinterseq}(H^{\anchor})=\frac{1}{\statinterseq_n}
\sum_{i=\dhinterseq+1}^{n-\dhinterseq}H^{\anchor}\left(\bsX_{i-\dhinterseq,i+\dhinterseq}/\orderstat[\norm{\bsX}]{n}{n-\statinterseq}\right)\;.
\end{align*}
The main result of this paper is \Cref{thm:sliding-block-clt-1}, the asymptotic normality of the appropriately normalized estimator
$\TEDclusterrandomr_{n,\dhinterseq}(H^{\anchor})$. We show, in particular, that the limiting variance agrees with the one for the disjoint blocks and sliding blocks estimators; cf. \cite{drees:rootzen:2010}, \cite[Chapter 10]{kulik:soulier:2020}, \cite{cissokho:kulik:2021}.
Furthermore, we prove that we cannot achieve variance reduction by considering a linear combination of runs estimators with a different choice of anchoring maps $\anchor$ and $\tilde\anchor$. Indeed, it turns out that $\tedclusterr(H^{\anchor})$ and $\tedclusterr(H^{\tilde\anchor})$ are totally dependent in the limit.
We note in passing that even though general ideas of proofs are similar to those of \cite{cissokho:kulik:2021}, however, technicalities are significantly different. Differences stem primarily from conditioning on $\{\norm{\bsX_j}>\tepseq\}$ used in case of the runs estimators.

Thus, from the theoretical point of view the limiting behaviour of all (disjoint blocks, sliding blocks, runs) estimators is the same. However, for finite samples a bias has to be taken into account. We note first that the theoretical finite-sample bias for both disjoint and sliding blocks estimators is the same. This can be also seen in extensive simulation studies in \cite{cissokho:kulik:2021}. On the other hand, we were not able to get an useful formula for the bias in the runs estimator case. As such we relied on simulations. It turns out that runs estimator are typically heavily biased when estimation of the extremal index is concerned. However, the runs estimators may have an advantage when other cluster indices are considered.

The paper is structured as follows. \Cref{sec:prel} contains definitions, notation and preliminary results on convergence of clusters. It is primarily based on \cite[Chapters 5 and 6]{kulik:soulier:2020}, with some results from \cite{basrak:segers:2009}, \cite{basrak:planinic:soulier:2018}, \cite{planinic:soulier:2018}. \Cref{sec:estimators} defines runs pseudo-estimators and estimators. The main result of the paper is the central limit theorem for runs estimators in
\Cref{thm:sliding-block-clt-1}. We note again that the limiting variance agrees with the one for disjoint and sliding blocks estimators. Simulations are performed in \Cref{sec:simulations}, while all the proofs are contained in \Cref{sec:proofs}.

\section{Preliminaries}\label{sec:prel}
In this section we fix the notation and introduce the relevant classes of functions. In \Cref{sec:tail-process} we recall the notion of the tail and the spectral tail process (cf. \cite{basrak:segers:2009}).
\Cref{sec:anchoring} introduces anchoring maps (cf. \cite{basrak:planinic:2018}, \cite{hashorva:2018}).
In \Cref{sec:cluster-index} we define cluster indices. We refer to \cite[Chapter 5]{kulik:soulier:2020} for more details. In \Cref{sec:weakconv-cluster-measure} we discuss convergence of the cluster measure, following
\cite[Chapter 6]{kulik:soulier:2020}.

The most important conclusion of these preliminaries is a representation of the cluster index $\tailmeasurestar(H)$ (cf. \eqref{eq:thelimitwhichisnolongercalledbH}) as $\esp[H^{\anchor}(\bsY)]$, with $H^{\anchor}$ defined in
\eqref{eq:HC-function} and $\bsY$ being the tail process. Also,  \Cref{lem:tailprocesstozero} on conditional weak convergence and \Cref{theo:cluster-RV,theo:limit-lone-Q}
on unconditional weak convergence play a central role in the rest of the paper.
\subsection{Notation}
Let $|\cdot|$ be a norm on $\Rset^d$.
For a sequence $\bsx=\{\bsx_j,j\in\Zset\}\in (\Rset^d)^\Zset$ and $i\leq j\in\Zset\cup\{-\infty,\infty\}$ we denote $\bsx_{i,j}=(\bsx_i,\ldots,\bsx_j)\in (\Rset^d)^{j-i+1}$, $\bsx_{i,j}^*=\max_{i\leq l\leq j}|\bsx_l|$ and $\bsx^*=\sup_{j\in\Zset}|\bsx_j|$. By $\bszero$ we denote the zero sequence; its dimension can be different in each of its occurrences.

By $\lzero(\Rset^d)$ we denote the set of $\Rset^d$-valued sequences which tend to zero at infinity. Likewise, $\lone(\Rset^d)$
consists of sequences such that $\sum_{j\in\Zset}|\bsx_j|<\infty$.


\subsection{Classes of functions}\label{sec:classes}
Functionals $H$ are defined on $\lzero(\Rset^d)$ with the convention $H(\bsx_{i,j})=H((\bszero,\bsx_{i,j},\bszero))$. For $s>0$, the function $H_s:(\Rset^d)^\Zset\to\Rset$ is defined by $H_s(\bsx)=H(\bsx/s)$.
We consider the following classes:
\begin{itemize}
	\item $\mcl$ is the class of bounded \realvalued\ functions defined on $(\Rset^d)^\Zset$ that are either
	Lipschitz continuous \wrt\ the uniform norm or almost
	surely continuous with respect to the distribution of the tail process $\bsY$.
	This class includes functions like $\ind{\bsx^*>1}$, $\ind{\sum_{j\in\Zset}|\bsx_j|>1}$.
	See Remark 6.1.6 in \cite{kulik:soulier:2020}.
	
	\item $\mca\subset\mcl$ is the class of \shiftinvariant\ functionals with support separated
	from $\bszero$. In particular, for $H\in \mca$, $H(\bszero)=0$. The class $\mca$ includes $\ind{\bsx^*>1}$.
	\item $\mck$ is the class of shift-invariant functionals $K:(\Rset^d)^\Zset\to\Rset$ defined on
	$\lone(\Rset^d)$ such that $K(\bszero)=0$ and which are
	Lipschitz continuous with constant $L_K$, \ie\,
	\begin{align*}
	|K(\bsx) - K(\bsy)| \leq L_K \sum_{j\in\Zset} \norm{\bsx_j-\bsy_j}\;, \ \ \bsx,\bsy\in\lone(\Rset^d) \; . 
	\end{align*}
	\item $\mcb\subset\mcl$ is the class of functionals $H$ of the form $H=\ind{K>1}$,
	where $K\in\mck$.
	Functionals in $\mcb$ may have support which is not
	separated from $\bszero$. The typical example is $H(\bsx)=\ind{\sum_{j}|\bsx_j|>1}$; note that $H\not\in\mca$.
\end{itemize}
We will also need
the map $\exc$ is defined on $\lzero(\Rset^d)$ by
$\exc(\bsx) = \sum_{j\in\Zset} \ind{\norm{\bsx_j}>1}$. Note that $\exc$ is shift-invariant, with the support separated from zero, but is not bounded.

\subsection{Tail and spectral tail process}\label{sec:tail-process}
Let $\bsX=\sequence{\bsX}$ be a stationary, regularly varying time series with values in $\Rset^d$ and tail index $\alpha$. In particular,
\begin{align*}
\lim_{x\to\infty}\frac{\pr(|\bsX_0|> tx)}{\pr(|\bsX_0|> x)}=t^{-\alpha}
\end{align*}
for all $t>0$.
Then, there exists a sequence $\bsY=\sequence{\bsY}$ such that
\begin{align*}
\pr(x^{-1}(\bsX_i,\dots,\bsX_j) \in \cdot \mid |\bsX_0|>x)   \mbox{ converges weakly to }  \pr((\bsY_i,\dots,\bsY_j) \in \cdot)
\end{align*}
as $x\to\infty$ for all $i\leq j\in\Zset$.
We call $\bsY$ the tail process.
See \cite{basrak:segers:2009}.
We note that, in particular, $|\bsY_0|$ has Pareto distribution with the density $\alpha x^{-\alpha-1}$, $x>1$. As such, it follows automatically that $\bsY^*=\sup_{j\in\Zset}|\bsY_j|>1$.
Equivalently, viewing $\bsX$ and $\bsY$ as random elements with values in $(\Rset^d)^\Zset$, we have for every bounded or
\nonnegative\ functional~$H$ on $(\Rset^d)^{\Zset}$, continuous \wrt\ the product topology,
\begin{align*}
\lim_{x\to\infty} \frac{\esp[H(x^{-1}\bsX)\ind{\norm{\bsX_0}>x}]} {\pr(\norm{\bsX_0}>x)}
& =  \esp[H(\bsY)] \; .
\end{align*}
The spectral tail process $\sequence{\bsTheta}$ is defined by $\bsTheta=\norm{\bsY_0}^{-1}\bsY$ and is independent of the tail process $\bsY$.
\subsection{Anchoring maps}\label{sec:anchoring}
\begin{definition}[Anchoring map]
	A measurable map $\anchor:(\Rset^d)^\Zset\to \Zset\cup\{-\infty,\infty\}$ is called an anchoring map if the following two properties hold:
	\begin{itemize}
		\item[{\rm An(i)}:] $\anchor(\bsx)=j$ implies $\norm{\bsx_j}\geq \norm{\bsx_0}\wedge 1$;
		\item[{\rm An(ii)}:]  $\anchor(B\bsx)=\anchor(\bsx)+1$, where $B$ is a backsift operator.
	\end{itemize}
\end{definition}
Three basic examples of anchoring maps are:
\begin{itemize}
	\item The infargmax functional:
	$\anchor^{(0)}(\bsy)=\inf\{j:\bsy_{-\infty,{j}}^*=\bsy^*\}$;
	\item
	The first exceedence above one: $\anchor^{(1)}(\bsx)=\inf\{j:\norm{\bsx_j}>1\}$;
	\item
	The last exceedence above one: $\anchor^{(2)}(\bsx)=\sup\{j:\norm{\bsx_j}>1\}$.
\end{itemize}
In what follows we use the convention $\inf\emptyset=+\infty$. We note that $\anchor^{(0)}$ is $0$-homogeneous, while $\anchor_s^{(1)}(\bsx)=\anchor^{(1)}(\bsx/s)$ and $\anchor_s^{(2)}(\bsx)=\anchor^{(2)}(\bsx/s)$ are increasing and decreasing in $s$, respectively, but they are not $0$-homogeneous. This will play a role in the proofs.

A special importance is given to the time index 0.  In particular,
\begin{itemize}
	\item
	If $\anchor^{(0)}(\bsx)=0$, then $\bsx_{-\infty,-1}^*<\norm{\bsx_0}$ and $\bsx_{1,\infty}^*\leq \norm{\bsx_0}$;
	\item
	If $\anchor^{(1)}(\bsx)=0$, then $\bsx^*_{-\infty,-1}\leq 1$ and $\norm{\bsx_0}> 1$;
	\item
	If $\anchor^{(2)}(\bsx)=0$, then $\bsx^*_{1,\infty}\leq 1$ and $\norm{\bsx_0}> 1$.
\end{itemize}
Applying an anchoring map to a finite block, say $\bsx_{-r,r}$ with $r\in\Nset$, is equivalent to applying it $\bsx=(\bszero,\bsx_{-r,r},\bszero)$. For example, $\anchor^{(0)}(\bsx_{-r,r})=0$ means that $\bsx_{-r,-1}^*<\norm{\bsx_0}$ and $\bsx_{1,r}^*\leq \norm{\bsx_0}$. This in turn implies also that $\anchor^{(0)}(\bsx_{-s,s})=0$ for $0<s<r$. Similarly,
\begin{align*}
\anchor(\bsx_{-r,r})=0 \Rightarrow \anchor(\bsx_{-s,s})=0\;, \ \ 0<s<r
\end{align*}
for $\anchor=\anchor^{(1)},\anchor^{(2)}$. However, we do not know if this important property (used explicitly in the proofs) holds for any anchoring map. As such, in the paper we focus on the three anchoring maps introduced above.

Furthermore, note that An(ii) gives that
\begin{align}
\anchor(\bsx_{h-r,h+r})=h \Leftrightarrow \anchor(\backshift^{-h}\bsx_{-r,r})=0\;. \label{eq:anchor-shift}
\end{align}
Indeed, consider for example $\anchor^{(0)}$. Then $\anchor^{(0)}(\bsx_{h-r,h+r})=h$ means that
$\bsx^*_{h-r,{h}-1}<\norm{\bsx_h}$ and $\bsx_{h+1,h+r}\leq \norm{\bsx_h}$. Set $\tilde\bsx=\backshift^{-h}\bsx$, so that $\tilde\bsx_{-r}=\bsx_{h-r}$, $\tilde\bsx_0=\bsx_h$ and $\tilde\bsx_r=\bsx_{h+r}$. Thus, $\tilde\bsx_{-r,-1}<\norm{\tilde\bsx_0}$ and $\tilde\bsx_{1,r}\leq\norm{\tilde\bsx_0}$. This in turn means that
$\anchor^{(0)}(\widetilde\bsx_{-r,r})=\anchor^{(0)}(\backshift^{-h}\bsx_{-r,r})=0$.

\subsection{Cluster measure and cluster indices}\label{sec:cluster-index}
Let $\anchor$ be an anchoring map.
If  $\pr(\anchor(\bsY)\notin\Zset)=0$ then we can define
\begin{align}
\label{eq:canditheta-anchor}
\canditheta = \pr(\anchor(\bsY)=0) \; .
\end{align}
We want to emphasize that $\canditheta$ does not depend on the choice of
the anchoring map (see \cite{planinic:soulier:2018} and \cite[Theorem 5.4.2]{kulik:soulier:2020}). In particular,
\begin{align*}
\canditheta=\pr(\anchor^{(1)}(\bsY)=0)=\pr\left(\sup_{j\leq -1}|\bsY_j|\leq 1\right)=\pr(\anchor^{(2)}(\bsY)=0)=\pr\left(\sup_{j\geq 1}|\bsY_j|\leq 1\right)\;.
\end{align*}
The above identity follows from the time-change formula, see \cite[Section 7.1]{cissokho:kulik:2021}.
Therefore, $\canditheta$ can be recognized as the (candidate) extremal index. It becomes the usual extremal index under additional mixing and anticlustering conditions (cf. Section 7.5 in \cite{kulik:soulier:2020}).

Recall that $\exc(\bsx)=\sum_{j\in\Zset}\ind{\norm{\bsx_j}>1}$.
The property An(i) of the anchoring maps implies
\begin{align}\label{eq:summability-of-Y}
\sum_{h\in\Zset}\pr(\anchor(\bsY)=h)\leq \sum_{h\in\Zset}\pr(\norm{\bsY_h}>1)\;.
\end{align}
By \cite[Lemma 9.2.3]{kulik:soulier:2020} the latter series is finite if an appropriate anticlustering condition holds (see \ref{eq:conditionS} to be introduced later on).
\begin{definition}[Cluster measure]
	\label{def:clustermeasure}
	Let $\bsY$ and $\bsTheta$ be the tail process and the spectral tail process, respectively, such that $\pr(\lim_{|j|\to\infty} \bsY_j=\bszero)=1$. The {cluster measure} is the measure $\tailmeasurestar$ on
	$\lzero(\Rset^d)$ defined by
	\begin{align}
	\label{eq:def-tailmeasurestar-premiere}
	\tailmeasurestar    = \canditheta \int_0^\infty \esp[\delta_{r\bsTheta}\ind{\anchor(\bsTheta)=0}] \alpha r^{-\alpha-1} \rmd r   \; .
	\end{align}
\end{definition}
The measure $\tailmeasurestar$ is
boundedly finite on $(\Rset^d)^\Zset\setminus\{\bszero\}$, puts no mass at $\bszero$ and is
$\alpha$-homogeneous.

Furthermore, the cluster measure can be expressed in terms of another sequence.
\begin{definition}
	\label{def:sequence-Q}
	Assume that $\pr(\anchor(\bsY)\notin\Zset)=0$. The \Qseq\ $\bsQ$ is a random sequence with the distribution of
	$(\bsY^*)^{-1}\bsY$ conditionally on $\anchor(\bsY)=0$.
\end{definition}
The sequence $\bsQ$ appeared implicitly in the seminal
paper \cite{davis:hsing:1995}.
See also \cite{basrak:segers:2009}, \cite[Definition 3.5]{planinic:soulier:2018} and \cite[Chapter 5]{kulik:soulier:2020}. An abstract setting is considered in \cite{dombry:hashorva:soulier:2018}.

Note for example that $\anchor^{(0)}(\bsY)=0$ gives $\bsY^*=\norm{\bsY_0}$. Thus, \eqref{eq:def-tailmeasurestar-premiere} and the definition of $\bsQ$ give for a bounded or \nonnegative\ measurable function $H$ on $\lzero(\Rset^d)$ (see Definition 5.4.11 in \cite{kulik:soulier:2020}),
\begin{align*}
\tailmeasurestar(H) & =
\vartheta \int_0^\infty \esp[H(r\bsQ)]  \alpha r^{-\alpha-1} \rmd r
\; .
\end{align*}
If moreover $H$ is such that $H(\bsy)=0$ if
$\bsy^*\leq\epsilon$ for one $\epsilon>0$, then
\begin{align}
\label{eq:relation-cluster-Y-Q-Theta-epsilon}
\tailmeasurestar(H) = \epsilon^{-\alpha} \esp[H(\epsilon\bsY) \ind{\anchor(\bsY)=0}] \; .
\end{align}
For a \shiftinvariant\ $H:(\Rset^d)^\Zset\to \Rset$ and an anchoring map $\anchor$ define
\begin{align}\label{eq:HC-function}
H^{\anchor}(\bsx)=H(\bsx)\ind{\anchor(\bsx)=0}\ind{\norm{\bsx_0}>1}\;.
\end{align}
Thus, since $\norm{\bsY_0}>1$, if $H$ is such that $H(\bsy)=0$ whenever
$\bsy^*\leq 1$, then \eqref{eq:relation-cluster-Y-Q-Theta-epsilon} gives
\begin{align}\label{eq:HA-on-Y}
\tailmeasurestar(H)=\esp[H^{\anchor}(\bsY)]\;.
\end{align}
Note that the $\tailmeasurestar(H)$ does not agree with $\esp[H(\bsY)]$.
\begin{definition}[Cluster index]
	\label{def:cluster-index}
	We will call $\tailmeasurestar(H)$ the cluster index associated to the
	functional $H$.
\end{definition}

\subsection{Convergence of cluster measure}\label{sec:weakconv-cluster-measure}
Define the measures $\tailmeasurestar_{n,\dhinterseq}$, $n\geq1$, on $\lzero(\Rset^d)$ as follows:
\begin{align*}
\tailmeasurestar_{n,\dhinterseq}
& =  \frac{1} {\dhinterseq \pr(\norm{\bsX_0}>\tepseq)} \esp \left[\delta_{\tepseq^{-1}\bsX_{1,\dhinterseq}} \right] \; .
\end{align*}
We are interested in convergence of $\tailmeasurestar_{n,\dhinterseq}$ to
$\tailmeasurestar$. The results of this section are extracted from \cite[Chapter 6]{kulik:soulier:2020}.
See also \cite{planinic:soulier:2018} and \cite{basrak:planinic:soulier:2018}.
\subsubsection{Anticlustering conditions}
For each fixed $r\in\Nset$, the distribution of
$\tepseq^{-1}\bsX_{-r,r}$ conditionally on $\norm{\bsX_0}>\tepseq$ converges weakly to the
distribution of $\bsY_{-r,r}$ (see \Cref{sec:tail-process}). In order to let $r$ tend to infinity, we must embed all these finite
vectors into one space of sequences. By adding zeroes on each side of the vectors
$\tepseq^{-1}\bsX_{-r,r}$ and $\bsY_{-r,r}$ we identify them with elements of the space
$\lzero(\Rset^d)$. Then $\bsY_{-r,r}$
converges (as $r\to\infty$) to $\bsY$ in $\lzero(\Rset^d)$ if (and only if) $\bsY\in\lzero(\Rset^d)$ almost surely.

However, this is not enough for statistical purposes and we consider the following definition.
\begin{definition}
	[\cite{davis:hsing:1995}, Condition~2.8]\label{def:DH}
	Condition~\ref{eq:conditiondh}
	holds if for all $x,y>0$,
	\begin{align}
	\label{eq:conditiondh}
	\lim_{m\to\infty} \limsup_{n\to\infty}\pr\left(\max_{m\le |j|\leq  \dhinterseq}|\bsX_j|
	> \tepseq x\mid |\bsX_0|> \tepseq y \right)=0 \; .
	\tag{$\conditiondh[\dhinterseq][\tepseq]$}
	\end{align}
\end{definition}
Condition \ref{eq:conditiondh} is referred to as the anticlustering condition.
It holds for \iid\ regularly varying sequences if $\lim_{n\to\infty}\dhinterseq \pr(\norm{\bsX_0}>\tepseq)=0$. Note that the latter condition is a part of the \ref{eq:rnbarFun0} assumption.
It is also fulfilled by many models, including geometrically ergodic Markov chains,
short-memory linear or max-stable processes.
\ref{eq:conditiondh} implies that $\bsY\in \lzero(\Rset^d)$. See \cite{kulik:soulier:wintenberger:2018} and \cite{kulik:soulier:2020}.

A stronger version of the anticlustering condition reads as follows.
\begin{definition}\label{def:conditionS}
	Condition
	\ref{eq:conditionS} holds if
	for all $s,t>0$
	\begin{align}
	\label{eq:conditionS}
	\lim_{m\to\infty} \limsup_{n\to\infty} \frac{1}{\pr(\norm{\bsX_0}>\tepseq)}
	\sum_{j=m}^{\dhinterseq} \pr(\norm{\bsX_0}>\tepseq s,\norm{\bsX_j}>\tepseq t) = 0 \;
	. \tag{$\conditiondhsumpaper{\dhinterseq}{\tepseq}$}
	\end{align}
\end{definition}
The
main consequence of the anticlustering condition \ref{eq:conditiondh} is the following result.
\begin{proposition}[\cite{basrak:segers:2009}, Proposition~4.2; \cite{kulik:soulier:2020}, Theorem 6.1.4]
	\label{lem:tailprocesstozero}
	Let $H\in \mcl$.
	If Condition~\ref{eq:conditiondh}
	holds, then
	\begin{align*}
	\lim_{n\to\infty} \esp[H(\tepseq^{-1} \bsX_{-\dhinterseq,\dhinterseq})\mid \norm{\bsX_0}>\tepseq] = \esp[H(\bsY)] \; .
	\end{align*}
\end{proposition}
\subsubsection{Vague convergence of cluster measure}
We now state the unconditional convergence of
$\tepseq^{-1}\bsX_{1,\dhinterseq}$. Contrary to
\Cref{lem:tailprocesstozero}, where an extreme value was imposed at time 0, a large value in the
cluster can happen at any time. Moreover, the convergence of $\tailmeasurestar_{n,\dhinterseq}(H)$ to
$\tailmeasure^*(H)$ may hold only for \shiftinvariant\ functionals $H$. Therefore, we need the following definition.
\begin{definition}
	\label{def:spacetildelzero}
	The space $\tildelzero(\Rset^d)$ is the space of equivalence classes of $\lzero(\Rset^d)$ endowed with the
	equivalence relation $\sim$ defined by
	\begin{align*}
	\bsx\sim \bsy \Longleftrightarrow \exists j \in\Zset \; , \ \shift^j\bsx=\bsy \; .
	\end{align*}
\end{definition}
\begin{proposition}[Theorem 6.2.5 in \cite{kulik:soulier:2020}]
	\label{theo:cluster-RV}
	Let  condition~\ref{eq:conditiondh} hold. The sequence of measures $\tailmeasurestar_{n,\dhinterseq}$, $n\geq1$ converges
	\vaguelysharp\ on $\tildelzero(\Rset^d)\setminus\{\bszero\}$ to~$\tailmeasurestar$, that is,
	for all $H\in \mca$,
	\begin{align*}
	\lim_{n\to\infty} \tailmeasurestar_{n,\dhinterseq}(H)
	= \lim_{n\to\infty}\frac{\esp[ H(\tepseq^{-1}\bsX_{1,\dhinterseq})]} {\dhinterseq \pr(\norm{\bsX_0}>\tepseq)}
	= \tailmeasure^*(H) \; .
	\end{align*}
\end{proposition}
The immediate consequence is the following limit (cf. \eqref{eq:canditheta-anchor}):
\begin{align*}
\lim_{n\to\infty} \frac{\pr(\bsX_{1,\dhinterseq}^*>\tepseq)}{\dhinterseq\pr(\norm{\bsX_0}>\tepseq)} = \canditheta  \; .
\end{align*}

\subsubsection{Indicator functionals not vanishing around zero}\label{sec:extension-1}
\Cref{theo:cluster-RV} entails convergence of $\tailmeasurestar_{n,\dhinterseq}(H)$ for
$H\in\mca$.
For functionals which are not
defined on the whole space $\lzero(\Rset^d)$ we need an additional assumption on Asymptotic Negligibility of Small Jumps.
\begin{definition}
	\label{def:AN-small}
	Condition~\ref{eq:AN-small} holds if for all
	$\eta>0$, \index{$\ANSJB(\dhinterseq,\tepseq)$}
	\begin{align}
	\lim_{\epsilon\to0} \limsup_{n\to\infty}
	\frac{\pr(\sum_{j=1}^{\dhinterseq} |\bsX_j| \ind{|\bsX_j|\leq\epsilon \tepseq}>\eta \tepseq)}
	{\dhinterseq\pr(|\bsX_0|>\tepseq)} = 0 \; . \tag{$\ANSJB(\dhinterseq,\tepseq)$} \label{eq:AN-small}
	\end{align}
\end{definition}
\begin{proposition}[Theorem 6.2.16 in \cite{kulik:soulier:2020}]
	\label{theo:limit-lone-Q}
	Assume that \ref{eq:conditiondh} and \ref{eq:AN-small} hold.   Then for $K\in{\mathcal K}$,
	\begin{align*}
	\tailmeasurestar(\ind{K>1})= \lim_{n\to\infty} \frac{\pr(K(\bsX_{1,\dhinterseq}/\tepseq )>1)} {\dhinterseq\pr(\norm{\bsX_0}>\tepseq )} = \canditheta
	\int_0^\infty \pr(K(z\bsQ)>1) \alpha z^{-\alpha-1} \rmd z  < \infty \; .  
	\end{align*}
\end{proposition}
\section{Central limit theorem for runs estimators}\label{sec:estimators}
In this section we introduce and study runs estimators of cluster indices. A pseudo-estimator is defined in \eqref{eq:process-xi}. Its limiting covariance (for different anchoring maps) is studied in \Cref{lem:covariance-runs}. In particular, for two different anchoring maps, the runs statistics are totally dependent. As a consequence we cannot reduce the limiting variance for the estimation of $\tailmeasurestar(H)$ by considering linear combinations of the runs statistics.
In \Cref{lem:covariance-runs-blocks} we consider covariance between runs and disjoint blocks estimators. Again, we obtain total dependence in the limit. The main result of the paper is the central limit theorem for runs estimators; see \Cref{thm:sliding-block-clt-1}.
The limiting variance agrees with the one for the disjoint blocks and sliding blocks estimators.

\subsection{Runs estimator}
\label{sec:runs}
To introduce runs estimators recall that (cf. \eqref{eq:anchor-shift})
\begin{align}\label{eq:shift-representation}
&H^{\anchor}\left(\bsX_{(j-1)\dhinterseq+h,(j+1)\dhinterseq+h}/\tepseq\right)=
H^{\anchor}(\backshift^{-h-j\dhinterseq}\bsX_{-\dhinterseq,\dhinterseq}/\tepseq)\nonumber\\
&=
H(\backshift^{-h-j\dhinterseq}\bsX_{-\dhinterseq,\dhinterseq}/\tepseq)
\ind{\anchor(\backshift^{-h-j\dhinterseq}\bsX_{-\dhinterseq,\dhinterseq}/\tepseq)=0}
\ind{\norm{\backshift^{-h-j\dhinterseq}\bsX_0}>\tepseq}\nonumber\\
&=H\left(\bsX_{(j-1)\dhinterseq+h,(j+1)\dhinterseq+h}/\tepseq\right)\times \nonumber\\
&\phantom{=}\ \ \ \ 
\ind{\anchor(\bsX_{(j-1)\dhinterseq+h,(j+1)\dhinterseq+h}/\tepseq)=h+j\dhinterseq}\ind{\norm{\bsX_{h+j\dhinterseq}}>\tepseq}\;.
\end{align}

Set $q_n=n-\dhinterseq$ and $m_n=n/\dhinterseq$. Without loss of generality assume that $m_n$ is an integer.
Consider disjoint blocks
\begin{align}\label{eq:disjoint-blocks}
J_j:=\{j\dhinterseq+1,\ldots,(j+1)\dhinterseq\}\;, \ \ j=0,\ldots,m_n-1\;.
\end{align}
The union of these blocks gives $\{1,\ldots,n\}$.
We assume we have data $\bsX_{1-\dhinterseq},\ldots,\bsX_{n+\dhinterseq}$.
For $j=0,\ldots,m_n-1$ define
\begin{align}
H_{n,j}^{\anchor}&=\sum_{i=j\dhinterseq+1}^{(j+1)\dhinterseq}
H^{\anchor}\left(\bsX_{i-\dhinterseq,i+\dhinterseq}/\tepseq\right)=
\sum_{i=j\dhinterseq+1}^{(j+1)\dhinterseq}
H^{\anchor}\left(B^{-i}\bsX_{-\dhinterseq,\dhinterseq}/\tepseq\right)\nonumber
\\
&=\sum_{i=j\dhinterseq+1}^{(j+1)\dhinterseq}
H\left(\bsX_{i-\dhinterseq,i+\dhinterseq}/\tepseq\right)\ind{\anchor\left(\bsX_{i-\dhinterseq,i+\dhinterseq}/\tepseq\right)=i}\ind{\norm{\bsX_i}>\tepseq}\;.\label{eq:Hnj}
\end{align}
Each $H_{n,j}^\anchor$ is a function of the block $\bsX_{(j-1)\dhinterseq+1,\ldots,(j+2)\dhinterseq}$ of size $3\dhinterseq$. The number $j$ in the notation  $H_{n,j}^\anchor$ indicates that the indicator $\norm{\bsX_i}>\tepseq$ is applied with $i\in J_j$.

We
consider a random process
\begin{align}
&\tedclusterr(H^{\anchor})=\frac{1}{n\pr(\norm{\bsX_0}>\tepseq)}
\sum_{i=1}^nH^{\anchor}\left(\bsX_{i-\dhinterseq,i+\dhinterseq}/\tepseq\right)\label{eq:process-xi}
\end{align}
that can be decomposed as
\begin{align*}
&\tedclusterr(H^{\anchor})=
\frac{1}{n\pr(\norm{\bsX_0}>\tepseq)}\sum_{j=0}^{m_n-1}H_{n,j}^{\anchor}\;.\nonumber
\end{align*}
If the anticlustering condition \ref{eq:conditiondh} holds, then
using stationarity, definition \eqref{eq:HC-function} of $H^\anchor$,
\Cref{lem:tailprocesstozero} and \eqref{eq:HA-on-Y}
we have
\begin{align*}
&\lim_{n\to\infty}\esp[\tedclusterr(H^{\anchor})]
=
\lim_{n\to\infty}
\frac{1}{\pr(\norm{\bsX_0}>\tepseq)}\esp\left[H^{\anchor}(\bsX_{-\dhinterseq,\dhinterseq}/\tepseq)\right]=\esp[H^{\anchor}(\bsY)]=\tailmeasurestar(H)\;.
\end{align*}
Now, let $\statinterseqn$ be a sequence of integers (depending on $n$) such that $\statinterseqn\to\infty$ and $\statinterseqn/n\to 0$. Define $\tepseq$ by $\statinterseqn=n\pr(\norm{\bsX_0}>\tepseq)$ and replace $\tepseq$ in $H^{\anchor}\left(\bsX_{i-\dhinterseq,i+\dhinterseq}/\tepseq\right)$ with $(\statinterseqn+1)$th order statistics $\orderstat[\norm{\bsX}]{n}{n-\statinterseqn}$ to get the
runs estimator:
\begin{align}\label{eq:process-xi-hat}
&\TEDclusterrandomr_{n,\dhinterseq}(H^{\anchor})=\frac{1}{\statinterseqn}
\sum_{i=1}^nH^{\anchor}\left(\bsX_{i-\dhinterseq,i+\dhinterseq}/\orderstat[\norm{\bsX}]{n}{n-\statinterseqn}\right)\;.
\end{align}

In what follows we will use interchangeably $\statinterseqn$ and $n\pr(\norm{\bsX_0}>\tepseq)$, whatever is more suitable.
\subsection{Mixing assumptions}

Dependence in $\sequence{\bsX}$ will be controlled by the $\beta$-mixing rates $\{\beta_n\}$.  Recall \ref{eq:rnbarFun0}.
Let $\{\ell_n\}$ be a sequence of integers such that $\lim_{n\to\infty}\ell_n=\infty$ and $\lim_{n\to\infty} \ell_n/\dhinterseq=0$.
\begin{definition}\label{def:condition-beta-r-l-sl}
	Condition $\beta'(\dhinterseq)$ holds if:
	\begin{subequations}
		\begin{align}
		&\lim_{n\to\infty}\frac{n}{\dhinterseq}\beta_{\dhinterseq}=0\;,\label{eq:mixing-rates-0}\\
		&\lim_{n\to\infty}
		\frac{1}
		{\pr(\norm{\bsX_0}>\tepseq)}
		\sum_{i=\dhinterseq+1}^{\infty}\beta_{i}=\lim_{n\to\infty}
		\frac{n}
		{\statinterseqn}
		\sum_{i=\dhinterseq+1}^{\infty}\beta_{i}=0\;,\label{eq:mixing-rates-runs}\\
		&\lim_{n\to\infty}\frac{1}{\dhinterseq \pr(\norm{\bsX_0}>\tepseq)}\sum_{j=1}^\infty \beta_{j\dhinterseq}=\lim_{n\to\infty}\frac{n}{\dhinterseq \statinterseqn}\sum_{j=1}^\infty \beta_{j\dhinterseq}=0\;.\label{eq:mixing-rates-7}
		\end{align}
	\end{subequations}
\end{definition}

\subsection{Limiting covariances}
\subsubsection{Runs statistics}
The first result deals with covariance of the process $\tedclusterr$ defined in \eqref{eq:process-xi}.
\begin{lemma}\label{lem:covariance-runs}
	Assume~\ref{eq:rnbarFun0},~\ref{eq:conditiondh},~\ref{eq:conditionS} and~\eqref{eq:mixing-rates-runs} hold. Let
	$H,\widetilde{H}\in \mcl$ and $\anchor,\widetilde\anchor$ be anchoring maps. Then
	\begin{align}\label{eq:covariance-statement}
	&\lim_{n\to\infty}n\pr(\norm{\bsX_0}>\tepseq)\cov\left(\tedclusterr(H^{\anchor}),\tedclusterr(\widetilde{H}^{\widetilde\anchor})\right)=\tailmeasurestar(H\widetilde{H})\;.
	\end{align}
\end{lemma}
We note that the limit does not depend on the choice of the anchoring maps. In other words, for two different anchoring maps, $\anchor$ and $\widetilde\anchor$, the runs statistics $\tedclusterr(H^{\anchor})$ and $\tedclusterr(H^{\widetilde\anchor})$ are totally dependent. As a consequence we cannot reduce the limiting variance for the estimation of $\tailmeasurestar(H)$ by considering a linear combination of $\tedclusterr(H^{\anchor})$ and $\tedclusterr(H^{\widetilde\anchor})$.

\subsubsection{Runs and disjoint blocks statistics}
We analyse covariance between
$\tedclusterr(H^{\anchor})$ defined in \eqref{eq:process-xi}
and the disjoint blocks statistics
\begin{align}\label{eq:blocks-estimator}
\frac{1}{n\pr(\norm{\bsX_0}>\tepseq)}\sum_{j=0}^{m_n-1}H\left(\bsX_{j\dhinterseq+1,(j+1)\dhinterseq}/\tepseq\right)
=\tedcluster(H)\;.
\end{align}
The disjoint blocks statistics are considered in \cite[Chapter 10]{kulik:soulier:2020}.
\begin{lemma}
	\label{lem:covariance-runs-blocks}
	Assume~\ref{eq:rnbarFun0},~\ref{eq:conditiondh},~\ref{eq:conditionS} and~\eqref{eq:mixing-rates-7} hold. Let
	$H,\widetilde{H}\in \mcl$, $\widetilde{H}(\bszero)=0$ and $\anchor$ be an anchoring map. Then
	\begin{align}\label{eq:covariance-statement-runs-blocks}
	&\lim_{n\to\infty}n\pr(\norm{\bsX_0}>\tepseq)\cov\left(\tedclusterr(H^{\anchor}),\tedcluster(\widetilde{H})\right)=\tailmeasurestar(H\widetilde{H})\;.
	\end{align}
\end{lemma}
Again, irrespectively of the choice of the anchoring map $\anchor$, the runs and disjoint blocks statistics are totally dependent and we cannot reduce the limiting variance by considering their linear combinations.

\subsection{Central limit theorem}

Let $\tepcluster$ be the Gaussian process on $L^2(\tailmeasurestar)$ with covariance
\begin{align*}
\cov(\tepcluster(H),\tepcluster(\widetilde{H})) =  \tailmeasurestar(H\widetilde{H}) \; .
\end{align*}
Recall that for a functional $H:(\Rset^d)^\Zset\to\Rset_+$ and $s>0$ we define $H_s(\bsx)=H(\bsx/s)$. Also, recall that $\exc(\bsx)=\sum_{j\in\Zset}\ind{|\bsx_j|>1}$.

Consider the class
$$\mcg = \{H^{\anchor}_s,s\in[s_0,t_0]\}=\{H(\bsx/s)\ind{\anchor(\bsx/s)=0}\ind{\norm{\bsx_0}>s},\;\;s\in[s_0,t_0]\}.$$

We need the following assumption on its random entropy.
\begin{hypothesis}\label{hypo:entropy}
	There exists a random metric $d_n$ on $\mcg$ and
	a measurable majorant $N^*(\mcg,d_n,\epsilon)$
	of the covering number $N(\mcg,d_n,\epsilon)$ such that for every sequence $\{\delta_n\}$ which
	decreases to zero,
	\begin{align}
	\int_0^{\delta_n} \sqrt{\log N^*(\mcg,d_n,\epsilon)} \rmd\epsilon\convprob 0  \; .
	\label{eq:randomentropy}
	\end{align}
\end{hypothesis}

The main result of this paper is \Cref{thm:sliding-block-clt-1}, the asymptotic normality of the appropriately normalized estimator
$\TEDclusterrandomr_{n,\dhinterseq}(H^{\anchor})$. The limiting variance agrees with the one for the disjoint blocks and sliding blocks estimators; cf. \cite{drees:rootzen:2010}, \cite[Chapter 10]{kulik:soulier:2020}, \cite{cissokho:kulik:2021}.
\begin{theorem}\label{thm:sliding-block-clt-1}
	Let $\sequence{\bsX}$ be a stationary, regularly varying $\Rset^d$-valued time series. Assume
	that~\ref{eq:rnbarFun0},
	$\beta'(\dhinterseq)$, \ref{eq:conditionS}
	and
	\begin{align}\label{eq:further-restriction-on-rn}
	\lim_{n\to\infty}\frac{\dhinterseq}{\sqrt{\statinterseqn}}=\lim_{n\to\infty}\frac{\dhinterseq}{\sqrt{n\pr(\norm{\bsX_0}>\tepseq)}}=0\;
	\end{align}
	hold. Suppose that \Cref{hypo:entropy} is satisfied.
	Fix $0<s_0<1<t_0<\infty$. Assume moreover that for $\anchor=\anchor^{(0)},\anchor^{(1)},\anchor^{(2)}$,
	\begin{subequations}
		\begin{align}
		\label{eq:bias-cluster-clt-aa}
		\lim_{n\to\infty}\sqrt{\statinterseqn}  \sup_{s\in [s_0,t_0]} \left|\frac{\pr(\norm{\bsX_0}>\tepseq s)}{\pr(\norm{\bsX_0}>\tepseq)}-s^{-\alpha}\right|
		&  = 0 \; , \\
		\label{eq:bias-cluster-clt-func}
		\lim_{n\to\infty}\sqrt{\statinterseqn}  \sup_{s\in [s_0,t_0]} |\esp[\tedclusterr(H_s^\anchor)] -\tailmeasurestar(H_s)| & = 0 \; .
		\end{align}
	\end{subequations}
	If $H\in\mca$, then
	\begin{align}\label{eq:clt-estimator-1} \sqrt{\statinterseqn}\left\{\TEDclusterrandomr_{n,\dhinterseq}(H^{\anchor})-\tailmeasurestar(H)\right\}\convdistr
	\tepcluster(H-\tailmeasurestar(H)\exc)\;.
	\end{align}
	If moreover \ref{eq:AN-small} is satisfied, then \eqref{eq:clt-estimator-1} holds for $H\in\mcb$.
\end{theorem}
\subsubsection{Comments on the conditions}
One chooses typically $\statinterseqn=n^{\epsilon}$ with some $\epsilon\in (0,1)$.
We note that $\beta'(\dhinterseq)$ holds if e.g. $\beta_n=O(n^{-\delta})$ with $\delta>1$ big enough or if $\beta_n$ decays logarithmically. In the latter case, we typically choose $\dhinterseq=(\log n)^{1+\delta}$ with some $\delta>0$. Recalling the choice of $\statinterseqn$ we can see that
\eqref{eq:further-restriction-on-rn} is not a very stringent assumption.

Furthermore, \eqref{eq:bias-cluster-clt-aa} controls the bias in the tail empirical process and can be related to the classical second order assumptions.

\Cref{hypo:entropy} controls the size of the class $\mcg$. We are not able to provide a general set of conditions under which this condition is satisfied, however, we will verify it for virtually all functionals $H$ that appeared in the paper. See \Cref{sec:random-entropy}.


\section{Simulation study}\label{sec:simulations}
We conducted some simulations in order to study the finite sample performance of the runs estimators for selected cluster indices. We compare their performance with the disjoint and sliding blocks estimators (see \cite{cissokho:kulik:2021}). Recall that the limiting variances are the same for all estimators. We do not have theoretical formulas for bias. We note that the bias for disjoint and sliding blocks estimators is the same, but different for runs estimators. We present only a small portion of our simulation studies; the most important findings are summarized at the end of the section.
\subsection{Stationary AR process}
We start with a simple AR(1) process. For this process we have explicit formulas for all cluster indices. Samples of size $n=1000$ are generated from AR(1) with $\alpha=4$ and $\rho=0.5, 0.9$. We perform simulations for the classical extremal index as well as for the stop-loss index.

\paragraph{Extremal index.}
For AR(1) with $\rho>0$ the extremal index is $\theta=1-\rho^\alpha$; cf. \cite[p. 396]{kulik:soulier:2020}.

\begin{itemize}
	\item \Cref{Summary table-AR-extremal} includes the results for Monte Carlo simulation for the extremal index based on disjoint blocks, sliding blocks and runs estimators with the block size $r_n=8,\;9$. We used $k=5\%$ and $10\%$ order statistics. We note that for the strong dependence ($\rho=0.9$), the sliding and disjoint blocks estimators outperform runs estimators for all considered parameters. For weak dependence ($\rho=0.5$), the results are heavily biased for all considered parameters. We note that all estimators yield almost the same variances, which is in agreement with the theoretical results obtained in the paper.
\end{itemize}
We note that the fact that stronger dependence yields smaller variability of the estimators is not surprising, cf. e.g. Figure 5 in \cite{robert:segers:ferro:2009}.

\paragraph{Stop-loss index.} For AR(1) with $\rho>0$ the formula for the stop-loss index is given in \cite[p. 619]{kulik:soulier:2020}:
\begin{align}\label{eq:stop-loss}
\theta_{{\rm stop-loss}}(S)=(1-\rho^{\alpha})\pr\left(\sum_{j=0}^\infty(\rho^j Y_0-1)_+>S\right)\;,
\end{align}
where $Y_0$ is a Pareto random variable with the parameter $\alpha$.

\begin{itemize}
	\item At the first step we use the formula \eqref{eq:stop-loss} and performed the Monte-Carlo simulation to obtain the approximate value of the stop-loss index.
	\item With this in mind, we performed simulation studies for $k=10\%$ and $k=40\%$. As noted in \cite{cissokho:kulik:2021}, the stop-loss index estimation requires a higher number of order statistics.
	We notice then (see \Cref{Summary table stop loss}) that, as opposed to the extremal index,  the runs estimator with the anchoring map $C^{(0)}$ performs better than the disjoint and the sliding blocks in case of weaker dependence. Indeed, the weaker dependence ($\rho=0.5$) yields a good estimation for runs estimators for any given block size, while for the strong dependence ($\rho=0.9$), the simulation results are rather poor for all the estimators. This may be quite intuitive, since the stop-loss functional is based on \textit{sums} of large values. On the other side, all the estimators perform poorly for strong dependence ($\rho=0.9$) as a result of bias.
	\item The box plots in \Cref{plot:box plots} is based again on Monte Carlo simulations. The following parameters are used: $\rho=0.5$, $\alpha=4$ and the block size $r_n=8,\;9$ along with $k=10\%$ and $k=40\%$. We notice again that in both cases the runs estimator yields acceptable result as opposed to disjoint and sliding blocks estimators.
\end{itemize}

\subsection{Stationary ARCH process}
We consider a stationary ARCH(1) process defined by $X_j^2=\sqrt{\beta+\lambda X_{j-1}^2}Z_j$, where $\{Z_j,j\in\Zset\}$ are i.i.d standard normal random variables. For $\lambda=0.9$ the extremal index is $\theta=0.612$ (see \cite[p. 480]{embrechts:kluppelberg:mikosch:1997}).
\begin{itemize}
	\item Monte Carlo results are included in Table \ref{Summary table-ARCH-extremal}. In this case both disjoint and sliding blocks estimators yield better results as compared to runs. This is primarily due to bias.
\end{itemize}

In summary, \textbf{
	\begin{itemize}
		\item All estimators (blocks and runs) have the same variance, which is in line with the theoretical results.
		\item
		For the extremal index, runs estimators are inferior as compared to blocks estimators. This is primarily due to bias.
		\item
		For the stop-loss index, runs estimators are superior, yielding lower (simulated) bias as compared to the blocks estimators.
	\end{itemize}
}


\section{Proofs}\label{sec:proofs}
In \Cref{sec:conditional-convergence} we prove several lemmas on conditional convergence when anchoring maps are involved. One needs to distinguish between finite blocks (when the conditional convergence follows basically from the conditional convergence to the tail process) and growing blocks (when the anticlustering condition is needed).

In \Cref{sec:covariances} we prove the asymptotic behaviour of the covariances of runs estimators, that is we prove
\Cref{lem:covariance-runs} and \Cref{lem:covariance-runs-blocks}.
\Cref{sec:fclt} deals with the empirical cluster process of runs statistics.
The functional central limit theorem (\Cref{thm:sliding-blocks-process-clt}) established there yields immediately the central limit theorem for runs estimators. See \Cref{sec:clt}.
A long proof of \Cref{thm:sliding-blocks-process-clt} is given in \Cref{sec:fidi,sec:tightness}. Finally, in \Cref{sec:random-entropy} we discuss the random entropy assumption.
\subsection{Mixing}
We recall the covariance inequality for bounded, beta-mixing random variables (in fact, the inequality holds for $\alpha$-mixing).
Let $\beta({\mcf_1},{\mcf_2})$ be the $\beta$-mixing coefficient between two sigma fields.
Then (cf. \cite{ibragimov:1962})
\begin{align}\label{eq:davydov-2}
|\cov(H(Z_1),H(Z_2))|\leq \constant\ \|H\|_{\infty} \|\widetilde{H}\|_{\infty}\beta(\sigma(Z_1),\sigma(Z_2))\;.
\end{align}
In \eqref{eq:davydov-2} the constant $\constant$ does not depend on $H,\widetilde{H}$.

\subsection{Conditional convergence}\label{sec:conditional-convergence}
\begin{lemma}\label{lem:fidi-convergence-anchoring}
	Let $\anchor=\anchor^{(0)},\anchor^{(1)},\anchor^{(2)}$. Then for $h\ \in\Zset$,
	\begin{align*}
	\lim_{n\to\infty}\pr(\anchor(\bsX_{h-r,h+r}/\tepseq)=h\mid \norm{\bsX_0}>\tepseq)=\pr(\anchor(\bsY_{h-r,h+r})=h)\;.
	\end{align*}
\end{lemma}
\begin{proof}
	We will do the proof for
	$\anchor^{(0)}$ only. By the definition of the tail process
	\begin{align*}
	&\lim_{n\to\infty}\pr(\anchor^{(0)}(\bsX_{h-r,h+r}/\tepseq)=h\mid \norm{\bsX_0}>\tepseq)\\
	&=\lim_{n\to\infty}\pr(\bsX_{h-r,h-1}^*/\tepseq<\norm{\bsX_h}/\tepseq,\bsX_{h+1,h+r}^*/\tepseq\leq\norm{\bsX_h}/\tepseq\mid \norm{\bsX_0}>\tepseq)\\
	&= \pr(\bsY_{h-r,h-1}^*<\norm{\bsY_h},\bsY_{h+1,h+r}^*\leq\norm{\bsY_h})=\pr(\anchor^{(0)}(\bsY_{h-r,h+r})=h)\;.
	\end{align*}
\end{proof}
\begin{lemma}\label{lem:fidi-convergence-anchoring-twodim}
	Let $\anchor,\widetilde{\anchor}$ be any of the anchoring maps $\anchor^{(0)},\anchor^{(1)},\anchor^{(2)}$. Then for $h\ \in\Zset$,
	\begin{align}\label{eq:joint-prob-anchor}
	&\lim_{n\to\infty}\pr(\anchor(\bsX_{-r,r}/\tepseq)=0,\widetilde\anchor(\bsX_{h-r,h+r}/\tepseq)=h\mid \norm{\bsX_0}>\tepseq)
	\nonumber\\
	&=
	\pr(\anchor(\bsY_{-r,r})=0,\widetilde\anchor(\bsY_{h-r,h+r})=h)\;.
	\end{align}
\end{lemma}
\begin{proof}
	We verify the statement for one combination of the anchoring maps only.
	For  $\anchor^{(1)}$ and  $\anchor^{(2)}$ , we have
	\begin{align*}
	&\lim_{n\to\infty}\pr(\anchor^{(1)}(\bsX_{h-r,h+r}/\tepseq)=h,\anchor^{(2)}(\bsX_{-r,r}/\tepseq)=0\mid \norm{\bsX_0}>\tepseq)\\
	&=\pr(\bsX_{h-r,h-1}^*\leq \tepseq,\norm{\bsX_h}>\tepseq, \bsX_{1,r}^*\leq \tepseq,\norm{\bsX_0}>\tepseq\mid \norm{\bsX_0}>\tepseq )\\
	&= \pr(\bsY_{h-r,h-1}^*\leq 1,\norm{\bsY_h}>1,\bsY_{1,r}^*\leq 1,\norm{\bsY_0}>1)\;.\\
	&=\pr(\anchor^{(1)}(\bsY_{h-r,h+r})=h,\anchor^{(2)}(\bsY_{-r,r})=0)\;.
	\end{align*}	
\end{proof}

Recall the definition of $H^\anchor$ in \eqref{eq:shift-representation}.
Let $h_n$ be a sequence of integers diverging to infinity. 
For bounded $H,\widetilde{H}$,
a direct application of \ref{eq:conditiondh} gives
\begin{align*}
&\lim_{n\to\infty}
\esp\left[H^{\anchor}\left(\bsX_{-\dhinterseq,\dhinterseq}/\tepseq\right)
\widetilde{H}^{\widetilde{\anchor}}\left(\bsX_{h_n-\dhinterseq,h_n+\dhinterseq}/\tepseq\right)
\mid \norm{\bsX_0}>\tepseq\right]\nonumber\\
&\leq \|H\|\|\widetilde{H}\| \lim_{n\to\infty}
\pr(\norm{\bsX_{h_n}}>\tepseq\mid \norm{\bsX_0}>\tepseq)=0\;.
\end{align*}
Likewise, if $H,\widetilde{H}$ are bounded and $\widetilde{H}\in\mcl$ is such that $\widetilde{H}(\bszero)=0$, then the Lipschitz continuity of $\widetilde{H}$ and \ref{eq:conditiondh} imply
\begin{align}\label{eq:covariance-vanishes-HA-H}
&\lim_{n\to\infty}
\esp\left[H^{\anchor}\left(\bsX_{-\dhinterseq,\dhinterseq}/\tepseq\right)
\widetilde{H}\left(\bsX_{h_n,h_n+\dhinterseq}/\tepseq\right)
\mid \norm{\bsX_0}>\tepseq\right]\nonumber\\
&\leq \|H\| \lim_{n\to\infty}
\esp\left[\widetilde{H}\left(\bsX_{h_n,h_n+\dhinterseq}/\tepseq\right)-\widetilde{H}(\bszero)\mid \norm{\bsX_0}>\tepseq\right]\nonumber \\
&\leq
\|H\| \|\widetilde{H}\|\lim_{n\to\infty}
\pr\left(\bsX_{h_n,h_n+\dhinterseq}^*>\tepseq\mid \norm{\bsX_0}>\tepseq\right)
=0\;.
\end{align}
The statement in \eqref{eq:covariance-vanishes-HA-H} is also valid if $\widetilde{H}\left(\bsX_{h_n,h_n+\dhinterseq}/\tepseq\right)$ is replaced with
$\widetilde{H}\left(\bsX_{-h_n,-h_n-\dhinterseq}/\tepseq\right)$.

On the other hand, for fixed $h$ we have the following lemma that extends
\Cref{lem:fidi-convergence-anchoring-twodim} from fixed $r$ to $r_n\to\infty$. For this, we need to assume additionally that \ref{eq:conditiondh} holds.
\begin{lemma}\label{lem:cond-conv-H}
	Assume that \ref{eq:conditiondh} holds. Let
	$H,\widetilde{H}\in \mcl$ and $\anchor,\widetilde{\anchor}$ be anchoring maps.
	Then
	\begin{align}\label{eq:limiting-covariance-fixed-h}
	&\lim_{n\to\infty}
	\esp\left[H^{\anchor}\left(\bsX_{-\dhinterseq,\dhinterseq}/\tepseq\right)
	\widetilde{H}^{\widetilde\anchor}\left(\bsX_{h-\dhinterseq,h+\dhinterseq}/\tepseq\right)
	\mid \norm{\bsX_0}>\tepseq\right]\nonumber\\
	&= \esp\left[H(\bsY)\widetilde{H}(\bsY)\ind{\anchor(\bsY)=0}\ind{\widetilde{\anchor}(\bsY)=h}\right]=:{\mathcal I}(H,\widetilde{H},\anchor,\widetilde{\anchor};h)\;.
	\end{align}
\end{lemma}
Before we prove the above lemma, we make several comments.

First, as a corollary we obtain
\begin{align}\label{eq:1dim-conv}
\lim_{n\to\infty}
\esp\left[H^{\anchor}\left(\bsX_{-\dhinterseq,\dhinterseq}/\tepseq\right)
\mid \norm{\bsX_0}>\tepseq\right]=\tailmeasurestar(H)\;
\end{align}
and
\begin{align}\label{eq:1dim-conv-anchor}
\lim_{n\to\infty}
\pr(\anchor(\bsX_{-\dhinterseq,\dhinterseq}/\tepseq)=0\mid \norm{\bsX_0}>\tepseq)=\tailmeasurestar(1)=1\;.
\end{align}
Indeed, if we take $\widetilde{H}\equiv 1$, $\anchor=\widetilde{\anchor}$ and $h=0$, then by
\eqref{eq:HA-on-Y},
\begin{align*}
{\mathcal I}(H,1,\anchor,{\anchor};0)&=
\esp\left[H(\bsY)\ind{\anchor(\bsY)=0}\right]\\
&=\esp\left[H(\bsY)\ind{\anchor(\bsY)=0}\ind{\norm{\bsY_0}>1}\right]
=\tailmeasurestar(H)\;.
\end{align*}

Since the definition of $\tailmeasurestar$ does not depend on the anchoring map, we have
${\mathcal I}(H,\widetilde{H},\anchor,\widetilde{\anchor};0)=\tailmeasurestar(H\widetilde{H})$ for any $\anchor,\widetilde{\anchor}$.
Since the value of any anchoring map is uniquely determined, we conclude immediately that  ${\mathcal I}(H,\widetilde{H},\anchor,{\anchor};h)=0$ for $h\not=0$.
Furthermore,
\begin{align*}
&\sum_{h\in\Zset}\esp\left[H(\bsY)\widetilde{H}(\bsY)\ind{\anchor(\bsY)=0}\ind{\widetilde{\anchor}(\bsY)=h}\right]\\
&=\esp\left[H(\bsY)\widetilde{H}(\bsY)\ind{\anchor(\bsY)=0}\ind{\widetilde{\anchor}(\bsY)\in\Zset}\right]=
\tailmeasurestar(H\widetilde{H})\;.
\end{align*}
This implies that for arbitrary anchoring maps $\anchor$, $\widetilde\anchor$,
\begin{align}\label{eq:covariance-zero}
{\mathcal I}(H,\widetilde{H},\anchor,\widetilde{\anchor};h)=0\;, \ \  h\not=0\;.
\end{align}
\begin{proof}[Proof of \Cref{lem:cond-conv-H}]
	In \cite{cissokho:kulik:2021} we proved a version of the lemma without anchoring maps included. Since $\bsx\to \ind{\anchor(\bsx)=0}$ is not Lipschitz continuous, Lemma 6.6 in \cite{cissokho:kulik:2021} is not directly applicable. As such, we will focus on the anchoring maps only, assuming $H=\widetilde{H}\equiv 1$.
	
	In the first step we prove that for all $h\in\Zset$
	\begin{align}\label{eq:step-1}
	\lim_{n\to\infty}\pr(\anchor(\bsX_{h-\dhinterseq,h+\dhinterseq}/\tepseq)=h\mid \norm{\bsX_0}>\tepseq)&=\pr(\anchor(\bsY)=h)\;.
	\end{align}
	We already know that (cf. \Cref{lem:fidi-convergence-anchoring})
	\begin{align*}
	\lim_{n\to\infty}\pr(\anchor(\bsX_{h-r,h+r}/\tepseq)=h\mid \norm{\bsX_0}>\tepseq)=\pr(\anchor(\bsY_{h-r,h+r})=h)\;.
	\end{align*}
	Since $\dhinterseq\to\infty$ and $r$ is fixed we can assume $0<r<\dhinterseq$. Now, for $\anchor=\anchor^{(0)},\anchor^{(1)}, \anchor^{(2)}$ the value of
	\begin{align*}
	\pr(\anchor(\bsY_{h-r,h+r})=h)-\pr(\anchor(\bsY_{h-\dhinterseq,h+\dhinterseq})=h)
	\end{align*}
	is \nonzero\ if and only if $h+r<\anchor(\bsY_{h-\dhinterseq,h+\dhinterseq})\leq h+\dhinterseq$ or $h-\dhinterseq\leq \anchor(\bsY_{h-\dhinterseq,h+\dhinterseq})<h-r$.
	Indeed, take for simplicity $h=0$. If $\anchor^{(1)}(\bsY_{-r,r})=0$ and $\anchor^{(1)}(\bsY_{-\dhinterseq,\dhinterseq})\not=0$, then
	$\bsY_{-r,-1}^*\leq 1$, $\norm{\bsY_0}>1$ and then $\bsY_{-\dhinterseq,-r-1}^*>1$, while
	$\anchor^{(1)}(\bsY_{-r,r})\not=0$ and $\anchor^{(1)}(\bsY_{-\dhinterseq,\dhinterseq})=0$ cannot happen. The same reasoning applied to the other anchoring maps.
	
	Coming back to the general case of $h$, the first property of the anchoring map implies that $\norm{\bsY_j}>1$ for some $j\in\{h+r+1,\ldots,h+\dhinterseq\}\cup
	\{h-\dhinterseq,\ldots,h-r-1\}$.
	Since we let $r,\dhinterseq\to\infty$, we can assume that $h<r$.
	Thus, using the property An(i) of the anchoring maps,
	\begin{align}\label{eq:approx-1}
	&\lim_{r\to\infty}{\lim_{n\to\infty}}|\pr(\anchor(\bsY_{h-r,h+r})=h)-\pr(\anchor(\bsY_{h-\dhinterseq,h+\dhinterseq})=h)|\nonumber\\
	&\leq
	\lim_{r\to\infty}{\lim_{n\to\infty}}
	\pr\left(\max\left\{\max_{h+r\leq j\leq h+\dhinterseq}\norm{\bsY_j},
	\max_{h-\dhinterseq\leq j\leq h-r}\norm{\bsY_j}\right\}> 1\right)=0
	\end{align}
	since \ref{eq:conditiondh} implies $\bsY_j\to \bszero$ almost surely as $|j|\to\infty$. Also, the vanishing property of $\bsY_j$ and the property An(i) of the anchoring map imply that
	\begin{align*}
	\lim_{n\to\infty}\pr(\anchor(\bsY_{h-\dhinterseq,h+\dhinterseq})=h)=\pr(\anchor(\bsY)=h)\;.
	\end{align*}

	Similarly,
	\begin{align*}
	\left|\pr(\anchor(\bsX_{h-\dhinterseq,h+\dhinterseq}/\tepseq)=h\mid \norm{\bsX_0}>\tepseq)-
	\pr(\anchor(\bsX_{h-r,h+r}/\tepseq)=h\mid \norm{\bsX_0}>\tepseq)\right|
	\end{align*}
	is \nonzero\ if and only if $h+r<\anchor(\bsX_{h-\dhinterseq,h+\dhinterseq}/\tepseq)\leq h+\dhinterseq$ or $h-\dhinterseq\leq \anchor(\bsX_{h-\dhinterseq,h+\dhinterseq}/\tepseq)<h-r$. The first property of the anchoring map implies that $\norm{\bsX_j}>\norm{\bsX_0}\wedge \tepseq$ for some $j\in\{h+r+1,\ldots,h+\dhinterseq\}\cup
	\{h-\dhinterseq,\ldots,h-r-1\}$. Again, we can assume that $h<r$. Keeping in mind the conditioning we have:
	\begin{align}\label{eq:approx-2}
	&\lim_{r\to\infty}\lim_{n\to\infty}\left|\pr(\anchor(\bsX_{h-\dhinterseq,h+\dhinterseq}/\tepseq)=h\mid \norm{\bsX_0}>\tepseq)-
	\pr(\anchor(\bsX_{h-r,h+r}/\tepseq)=h\mid \norm{\bsX_0}>\tepseq)\right|\nonumber\\
	&\leq \lim_{r\to\infty}\lim_{n\to\infty}\pr\left(\max\left\{\max_{h+r\leq j\leq h+\dhinterseq}\norm{\bsX_j},
	\max_{h-\dhinterseq\leq j\leq h-r}\norm{\bsX_j}\right\}
	>\tepseq\mid \norm{\bsX_0}>\tepseq\right)=0
	\end{align}
	by \ref{eq:conditiondh}. This finishes the proof of \eqref{eq:step-1}.

	Now, we will prove
	\begin{align}\label{eq:step-2}	&\lim_{n\to\infty}\pr(\anchor(\bsX_{-\dhinterseq,+\dhinterseq}/\tepseq)=0,\widetilde\anchor(\bsX_{h-\dhinterseq,h+\dhinterseq}/\tepseq)=h\mid \norm{\bsX_0}>\tepseq)\nonumber\\
	&=\pr(\anchor(\bsY)=0,\widetilde\anchor(\bsY)=h)\;.
	\end{align}
	In view of \Cref{lem:fidi-convergence-anchoring-twodim}, \eqref{eq:step-2} holds with $\dhinterseq$ replaced with $r$. Now, the idea is to reduce the bivariate case to the univariate.
	
	Note first that for the anchoring maps considered here, the event $A_1:=\{\anchor(\bsY_{h-\dhinterseq,h+\dhinterseq})=h\}$ is included in $A_2:=\{\anchor(\bsY_{h-r,h+r})=h\}$. We also note that for any event $B$ and any pair of ordered events $A_1,A_2$ we have
	$$
	|\pr(A_1\cap B)-\pr(A_2\cap B)|\leq |\pr(A_1)-\pr(A_2)|\;.
	$$
	Thus,
	we can bound
	\begin{align*}
	\left|\pr(\anchor(\bsY_{-r,r})=0,\widetilde\anchor(\bsY_{h-r,h+r})=h)-
	\pr(\anchor(\bsY_{-\dhinterseq,\dhinterseq})=0,\widetilde\anchor(\bsY_{h-\dhinterseq,h+\dhinterseq})=h)\right|
	\end{align*}
	by
	\begin{align*}
	\left|\pr(\anchor(\bsY_{-r,r})=0)-
	\pr(\anchor(\bsY_{-\dhinterseq,\dhinterseq})=0)\right|+
	\left|\pr(\widetilde\anchor(\bsY_{h-r,h+r})=h)-
	\pr(\widetilde\anchor(\bsY_{h-\dhinterseq,h+\dhinterseq})=h)\right|
	\end{align*}
	and we use the first step to conclude that
	\begin{align*}
	\lim_{r\to\infty}\lim_{n\to\infty}\left|\pr(\anchor(\bsY_{-r,r})=0,\widetilde\anchor(\bsY_{h-r,h+r})=h)-
	\pr(\anchor(\bsY_{-\dhinterseq,\dhinterseq})=0,\widetilde\anchor(\bsY_{h-\dhinterseq,h+\dhinterseq})=h)\right|=0\;.
	\end{align*}
	Therefore, \eqref{eq:approx-1} can be extended to the bivariate case. The same argument allows to extend \eqref{eq:approx-2} to the bivariate case. In summary, the proof of \eqref{eq:step-2} is finished.
\end{proof}
In the next lemma, we analyse the conditional convergence for the product of $H^{\anchor}$ and $\widetilde{H}$. Its proof is almost the same as above and hence it is omitted.
\begin{lemma}
	Assume that \ref{eq:conditiondh} holds. Let
	$H,\widetilde{H}\in \mcl$, $\widetilde{H}(\bszero)=0$ and $\anchor$ be an anchoring map.
	Then, for $h,h'\geq 0$,
	\begin{align}\label{eq:covariance-HA-H}
	&\lim_{n\to\infty}\esp\left[
	H^{\anchor}\left(\bsX_{-\dhinterseq,\dhinterseq}/\tepseq\right) \widetilde H\left(\bsX_{h-\dhinterseq,h'+\dhinterseq}/\tepseq\right)\mid \norm{\bsX_0}>\tepseq\right]\nonumber\\
	&=\esp\left[
	H\left(\bsY\right) \widetilde H\left(\bsY\right)\ind{\anchor(\bsY)=0}\right]=\tailmeasurestar(H\widetilde{H})
	\;.
	\end{align}
\end{lemma}

\subsection{Limiting Covariances}\label{sec:covariances}
The goal of this section is to prove \Cref{lem:covariance-runs,lem:covariance-runs-blocks}.
Two situations will arise when dealing with the covariances:
\begin{itemize}
	\item Situation 1: we will deal with $\sum_{h=-\dhinterseq}^{\dhinterseq}\esp[c_{h,n}(\bsX/\tepseq)]$, where
	\begin{align*} \lim_{n\to\infty}\esp[c_{h,n}(\bsX/\tepseq)]=\esp[c_h(\bsY)]\;, \ \ \sum_{h\in\Zset}\esp[|c_h(\bsY)|]<\infty\;.
	\end{align*}
	We will fix an integer $r>0$; the convergence of $\sum_{h=-r}^r \esp[c_{h,n}(\bsX/\tepseq)]$ to $\sum_{h=-r}^r\esp[c_h(\bsY)]$ will follow. The reminder $\sum_{|h|>r}\esp[c_h(\bsY)]$ is negligible (as $r\to\infty$) by the summability assumption, while $\sum_{h>|\dhinterseq|}\esp[c_{h,n}(\bsX/\tepseq)]$ will be treated by the anticlustering condition \ref{eq:conditionS}.
	
	\item Situation 2: we will deal with $\dhinterseq^{-1}\sum_{h=1}^{\dhinterseq}\esp[c_{h,n}(\bsX/\tepseq)]=\int_0^1 g_n(\xi)\rmd\xi$, where
	$g_n(h)=\esp[c_{h,n}(\bsX/\tepseq)]$ and $g_n(\xi)\to g(\xi)$ as $n\to\infty$. Bounded convergence argument will be applied.
\end{itemize}

\begin{proof}[Proof of \Cref{lem:covariance-runs}]
	Recall that
	\begin{align*}
	H_{n,j}^{\anchor}&=\sum_{i=j\dhinterseq+1}^{(j+1)\dhinterseq}
	H^{\anchor}\left(\bsX_{i-\dhinterseq,i+\dhinterseq}/\tepseq\right)\;.
	\end{align*}
	The covariance of the scaled statistics is
	\begin{align}\label{eq:covariance-calculation}
	&n\pr(\norm{\bsX_0}>\tepseq)\cov\left(\tedclusterr(H^{\anchor}),\tedclusterr(\widetilde{H}^{\widetilde\anchor})\right)=
	\frac{1}{\dhinterseq \pr(\norm{\bsX_0}>\tepseq)}
	\cov\left(H_{n,0}^\anchor,\widetilde{H}_{n,0}^{\widetilde\anchor}\right)\nonumber\\
	&\phantom{=}+\frac{1}{\dhinterseq \pr(\norm{\bsX_0}>\tepseq)}   \sum_{j=1}^{m_n-1}\left(1-\frac{j}{m_n}\right)\left\{\cov(H_{n,0}^{\anchor},\widetilde{H}_{n,j}^{\widetilde\anchor})
	+\cov(\widetilde{H}_{n,0}^{\widetilde{\anchor}},H_{n,j}^{\anchor})\right\}
	\;.
	\end{align}

	With the help of  \ref{eq:rnbarFun0}, we will show that
	$\cov\left(H_{n,0}^\anchor,\widetilde{H}_{n,0}^{\widetilde\anchor}\right)$ is determined by that of
	\begin{align}\label{eq:asymptotics-1}
	&\lim_{n\to\infty}
	\frac{1}{\dhinterseq\pr(\norm{\bsX_0}>\tepseq)}
	\cov\left(H_{n,0}^\anchor,\widetilde{H}_{n,0}^{\widetilde\anchor}\right)\nonumber\\
	&=
	\lim_{n\to\infty}\frac{1}{\pr(\norm{\bsX_0}>\tepseq)}\sum_{h=-\dhinterseq}^{\dhinterseq}
	\left(1-\frac{|h|}{\dhinterseq}\right)\esp[
	H^{\anchor}\left(\bsX_{0,2\dhinterseq}/\tepseq\right)\widetilde{H}^{\widetilde\anchor}\left(\bsX_{h,h+2\dhinterseq}/\tepseq\right)]
	\nonumber \\
	&=\tailmeasurestar(H\widetilde{H})\;.
	\end{align}
	We are in the Situation 1.
	For fixed $r$, using \eqref{eq:limiting-covariance-fixed-h} and \eqref{eq:covariance-zero} we have
	\begin{align}\label{eq:covariance-reduction}
	&\lim_{n\to\infty}\frac{1}{\pr(\norm{\bsX_0}>\tepseq)}\sum_{h=-r}^{r}\esp[
	H^{\anchor}\left(\bsX_{-\dhinterseq,\dhinterseq}/\tepseq\right)
	\widetilde{H}^{\widetilde\anchor}\left(\bsX_{h-\dhinterseq,h+\dhinterseq}/\tepseq\right)]\nonumber\\
	&=\sum_{h=-r}^{r}\esp[H(\bsY)\widetilde{H}(\bsY)\ind{\anchor(\bsY)=0}\ind{\widetilde\anchor(\bsY)=h}]=\sum_{h=-r}^{r}\mci(H,\widetilde{H},\anchor,\widetilde{\anchor},h)=
	\nonumber\\
	&=\mci(H,\widetilde{H},\anchor,\widetilde{\anchor},0)=\esp[H(\bsY)\widetilde{H}(\bsY)\ind{\anchor(\bsY)=0}]=\tailmeasurestar(H\widetilde{H})\;.
	\end{align}
	The value above does not depend on $r$. Moreover,
	\begin{align}\label{eq:consequence-of-condition-S}
	&\frac{1}{\pr(\norm{\bsX_0}>\tepseq)}\sum_{r<|h|\leq \dhinterseq}
	\esp[
	H^{\anchor}\left(\bsX_{-\dhinterseq,\dhinterseq}/\tepseq\right)\widetilde{H}^{\widetilde\anchor}\left(\bsX_{-\dhinterseq+h,\dhinterseq+h}/\tepseq\right)]\nonumber\\
	&\leq \|H\| \|\widetilde{H}\| \frac{1}{\pr(\norm{\bsX_0}>\tepseq)}\sum_{r<|h|\leq \dhinterseq}\pr(\norm{\bsX_0}>\tepseq,\norm{\bsX_h}>\tepseq)\;.
	\end{align}
	Letting $n\to\infty$ and then $r\to\infty$, we finish
	the proof of \eqref{eq:asymptotics-1} by applying \ref{eq:conditionS}.
	
	Now, we deal with the term in \eqref{eq:covariance-calculation}. For $j\geq 1$,
	\begin{align}\label{eq:covariance-term}
	&\frac{1}{\dhinterseq\pr(\norm{\bsX_0}>\tepseq)}
	\left|\cov(H_{n,0}^\anchor,\widetilde{H}_{n,j}^{\widetilde{\anchor}})\right|\nonumber\\
	&=
	\frac{1}{\dhinterseq\pr(\norm{\bsX_0}>\tepseq)}\left|\cov\left(\sum_{h=1}^{\dhinterseq}H^{\anchor} \left(\backshift^{-h}\bsX_{-\dhinterseq,\dhinterseq}/\tepseq\right),
	\sum_{i=1}^{\dhinterseq}\widetilde{H}^{\widetilde\anchor} \left(\backshift^{-i}\bsX_{(j-1)\dhinterseq,(j+1)\dhinterseq}/\tepseq\right)
	\right)\right|\nonumber\\
	&\leq\sum_{h=(j-1)\dhinterseq+1}^{j\dhinterseq}\left(\frac{h}{\dhinterseq}-(j-1)\right)|{g}_n(h)| +
	\sum_{h=j\dhinterseq+1}^{(j+1)\dhinterseq}\left((j+1)-\frac{h}{\dhinterseq}\right)|{g}_n(h)|\nonumber \\
	&
	\leq \sum_{h=(j-1)\dhinterseq+1}^{(j+1)\dhinterseq}|{g}_n(h)|=:I_{j}
	\end{align}
	with
	\begin{align*}
	{g}_n(h)&=\frac{1}{\pr(\norm{\bsX_0}>\tepseq)}\cov(H^{\anchor}(\bsX_{-\dhinterseq,\dhinterseq}/\tepseq),
	\widetilde{H}^{\widetilde\anchor}(\bsX_{h-\dhinterseq,h+\dhinterseq}/\tepseq))\;.
	\end{align*}
	For $h>2\dhinterseq$ we have by \eqref{eq:davydov-2},
	\begin{align}\label{eq:covariance-term-1}
	|{g}_n(h)|&\leq \frac{\|H\|_{\infty}\|\widetilde{H}\|_{\infty}}{\pr(\norm{\bsX_0}>\tepseq)} \beta_{h-2\dhinterseq}\;.
	\end{align}
	Thus,
	\begin{align*}
	&\frac{1}{\dhinterseq \pr(\norm{\bsX_0}>\tepseq)}   \sum_{j=4}^{m_n-1}\left|\cov(H_{n,0}^{\anchor},\widetilde{H}_{n,j}^{\widetilde\anchor})\right|\leq
	\frac{\|H\|_{\infty}\|\widetilde{H}\|_{\infty}}{\pr(\norm{\bsX_0}>\tepseq)}  \sum_{j=4}^{m_n-1}
	\sum_{h=(j-1)\dhinterseq+1}^{(j+1)\dhinterseq}\beta_{h-2\dhinterseq}\\
	&\leq 2\frac{\|H\|_{\infty}\|\widetilde{H}\|_{\infty}}{\pr(\norm{\bsX_0}>\tepseq)}  \sum_{h=3\dhinterseq+1}^{\infty}\beta_{h-2\dhinterseq}=O(1)\frac{1}{\pr(\norm{\bsX_0}>\tepseq)}
	\sum_{i=\dhinterseq+1}^{\infty}\beta_{i}=o(1)
	\end{align*}
	by the assumption \eqref{eq:mixing-rates-runs}.
	
	The terms that correspond to $j=1,2,3$ in \eqref{eq:covariance-calculation} have to be dealt with separately. We are again in the Situation 1. We have
	\begin{align*}
	I_1+I_2+I_3\leq 2\sum_{h=1}^{4\dhinterseq}|g_n(h)|=2\left\{ \sum_{h=1}^r+\sum_{i=r+1}^{4\dhinterseq}\right\} |g_n(h)|\;.\;
	\end{align*}
	Both parts are negligible. Indeed, as in \eqref{eq:covariance-reduction},
	\begin{align*}
	&\lim_{n\to\infty}\sum_{h=1}^r|g_n(h)|\leq \lim_{n\to\infty}\frac{1}{\pr(\norm{\bsX_0}>\tepseq)}\sum_{h=1}^{r}\esp[
	H^{\anchor}\left(\bsX_{-\dhinterseq,\dhinterseq}/\tepseq\right)
	\widetilde{H}^{\widetilde\anchor}\left(\bsX_{h-\dhinterseq,h+\dhinterseq}/\tepseq\right)]\nonumber\\
	&=\sum_{h=1}^{r}\esp[H(\bsY)\widetilde{H}(\bsY)\ind{\anchor(\bsY)=0}\ind{\widetilde\anchor(\bsY)=h}]=
	\sum_{h=1}^r\mci(H,\widetilde{H},\anchor,\widetilde{\anchor},h)
	\end{align*}
	and by \eqref{eq:covariance-zero} the last term vanishes.
	
	For the term $\sum_{h=r+1}^{4\dhinterseq}$ we apply \ref{eq:conditionS}; see the argument used in \eqref{eq:consequence-of-condition-S}.
	
	This finishes the proof of the lemma.
\end{proof}
\begin{proof}[Proof of \Cref{lem:covariance-runs-blocks}]
	Recall that
	\begin{align*}
	\widetilde{H}_j=
	H\left(\bsX_{j\dhinterseq+1,(j+1)\dhinterseq}/\tepseq\right)\;.
	\end{align*}
	Here, $\widetilde{H}_j$ is a function of the $j$th block $\bsX_{j\dhinterseq+1,(j+1)\dhinterseq}$, $j=0,\ldots,m_n-1$. Since $H_{n,j}^\anchor$, $j=0,\ldots,m_n-1$, is a function of the block $\bsX_{(j-1)\dhinterseq+1,\ldots,(j+2)\dhinterseq}$ (recall that we assumed that
	we have data $\bsX_{1-\dhinterseq},\ldots,\bsX_{n+\dhinterseq}$), for $|q|\geq 3$,
	\begin{align}\label{eq:mixing-bound}
	\cov(H_{n,j}^\anchor,\widetilde{H}_{j+q})\leq \|H\|\|\widetilde{H}\|\beta_{(|q|-2)\dhinterseq}\;;
	\end{align}
	cf. \eqref{eq:davydov-2}.
	We have
	\begin{align}\label{eq:covariance-calculation-disblocks}
	&\statinterseqn\cov\left(\tedclusterr(H^{\anchor}),\tedcluster(\widetilde H)\right)=
	\frac{1}{\dhinterseq \pr(\norm{\bsX_0}>\tepseq)}
	\cov\left(H_{n,0}^\anchor,\widetilde H_0\right)\nonumber\\
	&\phantom{=}+\frac{1}{\dhinterseq \pr(\norm{\bsX_0}>\tepseq)}   \sum_{j=1}^{m_n-1}\left(1-\frac{j}{m_n}\right)\left\{\cov(H_{n,0}^{\anchor},\widetilde H_{j})
	+\cov(\widetilde H_0,H_{n,j}^{\anchor})\right\}
	\;.
	\end{align}
	We analyse $\cov\left(H_{n,0}^\anchor,\widetilde H_0\right)$.
	
	We are in the Situation 2:
	\begin{align*}
	&\frac{1}{\dhinterseq \pr(\norm{\bsX_0}>\tepseq)}
	\esp\left[H_{n,0}^\anchor\widetilde H_0\right]
	\\
	&=\frac{1}{\dhinterseq \pr(\norm{\bsX_0}>\tepseq)}
	\sum_{i=1}^{\dhinterseq}\esp\left[
	H^{\anchor}\left(\bsX_{i-\dhinterseq,i+\dhinterseq}/\tepseq\right) \widetilde H\left(\bsX_{1,\dhinterseq}/\tepseq\right)\right]\\
	&=\frac{1}{\dhinterseq \pr(\norm{\bsX_0}>\tepseq)}
	\sum_{i=1}^{\dhinterseq}\esp\left[
	H\left(\bsX_{i-\dhinterseq,i+\dhinterseq}/\tepseq\right)
	\ind{\anchor\left(\bsX_{i-\dhinterseq,i+\dhinterseq}\right)=0}
	\ind{\norm{\bsX_i}>\tepseq} \widetilde H\left(\bsX_{1,\dhinterseq}/\tepseq\right)\right]\\
	&=\frac{1}{\dhinterseq \pr(\norm{\bsX_0}>\tepseq)}
	\sum_{i=1}^{\dhinterseq}\esp\left[
	H\left(\bsX_{-\dhinterseq,\dhinterseq}/\tepseq\right)\ind{\anchor\left(\bsX_{-\dhinterseq,\dhinterseq}\right)=0}\ind{\norm{\bsX_0}>\tepseq} \widetilde H\left(\bsX_{1-i,\dhinterseq-i}/\tepseq\right)\right]\\
	&=\frac{1}{\dhinterseq}
	\sum_{i=1}^{\dhinterseq}\esp\left[
	H^{\anchor}\left(\bsX_{-\dhinterseq,\dhinterseq}/\tepseq\right) \widetilde H\left(\bsX_{1-i,\dhinterseq-i}/\tepseq\right)\mid \norm{\bsX_0}>\tepseq\right]=\int_0^1h_{n,0}(\xi)\rmd\xi
	\end{align*}
	with
	\begin{align*}
	h_{n,0}(\xi)=\esp\left[
	H^{\anchor}\left(\bsX_{-\dhinterseq,\dhinterseq}/\tepseq\right) \widetilde H\left(\bsX_{1-[\xi\dhinterseq],\dhinterseq-[\xi\dhinterseq]}/\tepseq\right)\mid \norm{\bsX_0}>\tepseq\right]\;, \ \ \xi\in (0,1)\;.
	\end{align*}
	Note that the third equality follows by stationarity.
	By \eqref{eq:covariance-HA-H}, for each $\xi\in (0,1)$, $h_{n,0}(\xi)\to \tailmeasurestar(H\widetilde{H})$.
	Furthermore, the sequence $\{h_{n,0},n\geq 1\}$ is uniformly bounded in $n$ and $\xi$.
	Thus, with help of \ref{eq:rnbarFun0},
	\begin{align*}
	\lim_{n\to\infty}\frac{1}{\dhinterseq \pr(\norm{\bsX_0}>\tepseq)}\cov(H_{n,0}^\anchor,\widetilde H_0)=
	\lim_{n\to\infty}\frac{1}{\dhinterseq \pr(\norm{\bsX_0}>\tepseq)}\esp[H_{n,0}^\anchor\widetilde H_0]=
	\tailmeasurestar(H\widetilde{H})\;.
	\end{align*}
	The other covariances vanish. Indeed,
	we analyse $\cov(H_{n,0}^\anchor,\widetilde H_{j})$, $j\geq 1$.
	We have, using again the stationarity as above,
	\begin{align*}
	&\frac{1}{\dhinterseq \pr(\norm{\bsX_0}>\tepseq)}\esp[H_{n,0}^\anchor,\widetilde H_{j}]\\
	&=
	\frac{1}{\dhinterseq \pr(\norm{\bsX_0}>\tepseq)}
	\sum_{i=1}^{\dhinterseq}\esp\left[
	H^{\anchor}\left(\bsX_{i-\dhinterseq,i+\dhinterseq}/\tepseq\right) \widetilde H\left(\bsX_{j\dhinterseq+1,(j+1)\dhinterseq}/\tepseq\right)\right]\\
	&=\frac{1}{\dhinterseq \pr(\norm{\bsX_0}>\tepseq)}
	\sum_{i=1}^{\dhinterseq}\esp\left[
	H^{\anchor}\left(\bsX_{-\dhinterseq,\dhinterseq}/\tepseq\right) \widetilde H\left(\bsX_{j\dhinterseq+1-i,(j+1)\dhinterseq-i}/\tepseq\right)\right]\\
	&=\int_{0}^{1}h_{n,j}(\xi)\rmd\xi
	\end{align*}
	with a function $h_{n,j}$ defined on $(0,1)$ by
	\begin{align*}
	h_{n,j}(\xi)=\esp\left[
	H^{\anchor}\left(\bsX_{-\dhinterseq,\dhinterseq}/\tepseq\right) \widetilde H\left(\bsX_{j\dhinterseq-[\xi\dhinterseq]+1,(j+1)\dhinterseq-[\xi\dhinterseq]}/\tepseq\right)\mid \norm{\bsX_0}>\tepseq\right]\;.
	\end{align*}
	Until now we proceeded as in the case $j=0$ above. However, now we use
	\eqref{eq:covariance-vanishes-HA-H}. For each $\xi\in (0,1)$, $j\dhinterseq-[\xi\dhinterseq]\to +\infty$. Hence, $h_{n,j}(\xi)\to 0$. Bounded convergence and \ref{eq:rnbarFun0} give
	\begin{align}\label{eq:H-HA-1}
	\lim_{n\to\infty}\frac{1}{\dhinterseq \pr(\norm{\bsX_0}>\tepseq)}\cov(H_{n,0}^\anchor,\widetilde H_{j})=0\;.
	\end{align}
	The same idea applies to $\cov(\widetilde H_{0}, H_{n,j}^\anchor)$, $j\geq 1$:
	\begin{align}\label{eq:H-HA-2}
	\lim_{n\to\infty}\frac{1}{\dhinterseq \pr(\norm{\bsX_0}>\tepseq)}\cov(\widetilde H_{0}, H_{n,j}^\anchor)=0\;.
	\end{align}
	Now, by \eqref{eq:H-HA-1}-\eqref{eq:H-HA-2}, the terms that correspond to $j=1,2$ in  \eqref{eq:covariance-calculation-disblocks} vanish, while \eqref{eq:mixing-bound} and \eqref{eq:mixing-rates-7} give
	\begin{align*}
	&\frac{1}{\dhinterseq \pr(\norm{\bsX_0}>\tepseq)}\sum_{j=3}^{m_n-1}
	\left\{\cov(H_{n,0}^{\anchor},\widetilde H_{j})
	+\cov(\widetilde H_0,H_{n,j}^{\anchor})\right\}\\
	&=O(1)\frac{1}{\dhinterseq \pr(\norm{\bsX_0}>\tepseq)}\sum_{j=1}^\infty \beta_{j\dhinterseq}=o(1)\;.
	\end{align*}
\end{proof}
\subsection{Empirical cluster process of runs statistics}\label{sec:fclt}
Recall that
\begin{align*}
H^{\anchor}(\bsx)=H(\bsx)\ind{\anchor(\bsx)=0}\ind{\norm{\bsx_0}>1}\;.
\end{align*}
Define
\begin{align}\label{eq:HC with s}
H_s^{\anchor}(\bsx)=H(\bsx/s)\ind{\anchor(\bsx/s)=0}\ind{\norm{\bsx_0}>s}\;.
\end{align}
Recall that $0<s_0<1<t_0<\infty$. Recall also that $\statinterseqn=n\pr(\norm{\bsX_0}>\tepseq)$.
Define also the classical tail empirical process by
\begin{align*}
\mathbb{T}_n(s)=\sqrt{\statinterseqn}\left\{\frac{\sum_{j=1}^n\ind{\norm{\bsX_j}> s\tepseq}}{\statinterseqn}-s^{-\alpha}\right\}\;, \ \ s\in [s_0,t_0]\;.
\end{align*}

In order to deal with asymptotic normality of runs estimators, we study the empirical process
\begin{align*}
\mathbb{F}_n(H_s^{\anchor})&:=\sqrt{\statinterseqn}
\left\{
\tedclusterr(H_s^{\anchor})-\tailmeasurestar(H_s)\right\}\\
&
=\sqrt{\statinterseqn}
\left\{
\frac{\sum_{i=1}^{n}H_s^{\anchor}\left(\bsX_{i-\dhinterseq,i+\dhinterseq}/\tepseq\right)}{ \statinterseqn}-s^{-\alpha}\tailmeasurestar(H)\right\}\;.
\end{align*}
The process $\mathbb{F}_n(H_s^{\anchor})$ is viewed as a random element with values in
$\spaceD([s_0,t_0])$. The next result is crucial to establish convergence of runs estimators.
\begin{theorem}
	\label{thm:sliding-blocks-process-clt}
	Let $\sequence{\bsX}$ be a stationary, regularly varying $\Rset^d$-valued time series.
	Assume
	that~\ref{eq:rnbarFun0}, $\beta'(\dhinterseq)$, \ref{eq:conditionS},  \eqref{eq:further-restriction-on-rn} and \eqref{eq:bias-cluster-clt-func}
	hold. Suppose that \Cref{hypo:entropy} is satisfied. \\
	
	Then ${\mathbb{F}}_n(H^{\anchor}_{\cdot})$ converges weakly in $(\spaceD([s_0,t_0]),J_1)$ to a Gaussian process ${\mathbb{G}}(H_{\cdot})$ with the covariance
	$\tailmeasurestar(H_sH_t)$.
	If moreover \ref{eq:AN-small} is satisfied, then the convergence holds for $H\in\mcb$.
	If additionally \eqref{eq:bias-cluster-clt-aa} is satisfied, then
	the processes ${\mathbb{F}}_n(H^{\anchor}_{\cdot})$ and ${\mathbb{T}}_n(\cdot)$
	converge jointly $({\mathbb{G}}(H_{\cdot}),{\mathbb{G}}(\exc_{\cdot}))$.
\end{theorem}

\subsection{Proof of \Cref{thm:sliding-block-clt-1}}\label{sec:clt}
Write $\psi_n=\orderstat[\norm{\bsX}]{n}{n-\statinterseqn}/u_n$. Since $\statinterseqn=n\pr(\norm{\bsX_0}>\tepseq)$, we can rewrite $\TEDclusterrandomr_{n,\dhinterseq}(H^{\anchor})$ as
$\TEDclusterrandomr_{n,\dhinterseq}(H^{\anchor})= \tedclusterr (H^{\anchor}_{\psi_n})$ (cf. \eqref{eq:process-xi}-\eqref{eq:process-xi-hat}). Therefore,

\begin{align*}
\sqrt{\statinterseqn}\left\{\TEDclusterrandomr_{n,\dhinterseq}(H^{\anchor})-\tailmeasurestar(H)\right\}={\mathbb{F}}_n(H^{\anchor}_{\psi_n}) +\sqrt{\statinterseqn}\left\{\tailmeasurestar(H_{\psi_n})-\tailmeasurestar(H)\right\}\;.
\end{align*}
We have local uniform convergence of  $\{\mathbb{F}_n(H^{\anchor}_{s}),\;s\in [s_0,\,t_0]\}$ to a continuous Gaussian  process
$\tepcluster$ thanks to Theorem \ref{thm:sliding-blocks-process-clt}. Moreover, the convergence of $\{{\mathbb{T}}_n(\cdot),\;s\in [s_0,\,t_0]\}$ yields $\psi_n\convdistr 1$, jointly with $\mathbb{F}_n(H^{\anchor}_{s})$. Therefore, $\mathbb{F}_n(H^{\anchor}_{\psi_n})\convdistr
\tepcluster(H) $. Using Vervaat's theorem, we have, jointly with the previous convergence, $\sqrt{\statinterseqn}(\psi_n^{-\alpha}-1) \convdistr - \tepcluster(\exc)$. Therefore, by the homogeneity of $\tailmeasurestar$,
\begin{align*}
\sqrt{\statinterseqn}\left\{\tailmeasurestar(H_{\psi_n})-\tailmeasurestar(H)\right\} =\tailmeasurestar(H)\sqrt{\statinterseqn}(\psi_n^{-\alpha}-1)  \convdistr -\tailmeasurestar(H)\tepcluster(\exc).
\end{align*}
Since the convergence hold jointly, we conclude the result.

\subsection{Proof of \Cref{thm:sliding-blocks-process-clt} - fidi convergence}\label{sec:fidi}
Recall the disjoint blocks of size $\dhinterseq$ (cf. \eqref{eq:disjoint-blocks}):
\begin{align*}
J_j:=\{j\dhinterseq+1,\ldots,(j+1)\dhinterseq\}\;, \ \ j=0,\ldots,m_n-1\;.
\end{align*}
These blocks were chosen to calculate the limiting covariance of the process $\mathbb{F}_n$.
However, they are not appropriate for a proof of the central limit theorem. We need to introduce a large-small blocks decomposition.

For this purpose let $z_n$ be a sequence of integers such that $z_n\to\infty$
and
\begin{align}\label{eq:zn}
\lim_{n\to\infty}z_n\dhinterseq\pr(\norm{\bsX_0}>\tepseq)=
\lim_{n\to\infty}\frac{z_n\dhinterseq}{\sqrt{n\pr(\norm{\bsX_0}>\tepseq)}}=\lim_{n\to\infty}
\frac{z_n\dhinterseq}{\sqrt{\statinterseqn}}=0\;.
\end{align}
This is possible thanks to the assumptions \ref{eq:rnbarFun0} and \eqref{eq:further-restriction-on-rn}. We note that this assumption is needed for the Lindeberg condition only.
Set
$$\widetilde{m}_n=\frac{q_n}{(z_n+3)\dhinterseq}=\frac{n-\dhinterseq}{(z_n+3)\dhinterseq}\sim \frac{n}{z_n\dhinterseq}$$
and assume for simplicity that $\widetilde{m}_n$ is an integer. Since $z_n\to\infty$, we have $\widetilde{m}_n=o(m_n)$.
For $j=1,\ldots,\widetilde{m}_n$ define now large and small blocks as follows:
\begin{align*}
&L_1=\{1,\ldots, z_n\dhinterseq\}\;,\ \  S_1=\{z_n\dhinterseq+1\ldots,z_n\dhinterseq+3\dhinterseq\}\;, \\
&L_2=\{z_n\dhinterseq+3\dhinterseq+1,\ldots, 2z_n\dhinterseq+3\dhinterseq\}\;,\ \  S_2=\{2z_n\dhinterseq+3\dhinterseq+1,\ldots,2z_n\dhinterseq+6\dhinterseq\}\;, \\
&L_j=\{(j-1)z_n\dhinterseq+3(j-1)\dhinterseq+1,\ldots, jz_n\dhinterseq+3(j-1)\dhinterseq\}\;,\\ &S_j=\{jz_n\dhinterseq+3(j-1)\dhinterseq+1,\ldots,jz_n\dhinterseq+3j\dhinterseq\}\;.
\end{align*}
The block $L_1$ is obtained by merging $z_n$ consecutive blocks $J_0,\ldots,J_{z_n-1}$ of size $\dhinterseq$. Likewise, $S_1=J_{z_n}\cup J_{z_n+1}\cup J_{z_n+2}$.
Therefore, the large block of size $z_n\dhinterseq$ is followed by the small block of size $3\dhinterseq$, which in turn is followed by the large block of size $z_n\dhinterseq$ and so on. All together,
\begin{align*}
\bigcup_{j=1}^{\widetilde{m}_n}\left(L_j\cup S_j\right)=\{1,\ldots,q_n\}=\{1,\ldots,n-\dhinterseq\}\;.
\end{align*}
Write
\begin{align}\label{eq:large-small-block-process}
&
\sum_{i=1}^{n}H^{\anchor}\left(\bsX_{i-\dhinterseq,i+\dhinterseq}/\tepseq\right)=
\sum_{i=1}^{q_n}H^{\anchor}\left(\bsX_{i-\dhinterseq,i+\dhinterseq}/\tepseq\right)+
\sum_{i=q_n+1}^{n}H^{\anchor}\left(\bsX_{i-\dhinterseq,i+\dhinterseq}/\tepseq\right)\nonumber\\
&=\sum_{j=1}^{\widetilde{m}_n}\Psi_j^{(l)}(H^{\anchor})
+\sum_{j=1}^{\widetilde{m}_n}\Psi_j^{(s)}(H^{\anchor})+W_n\;,
\end{align}
where now
\begin{align*}
\Psi_j^{(l)}(H^{\anchor})=
\sum_{i\in L_j}H^{\anchor}\left(\bsX_{i-\dhinterseq,i+\dhinterseq}/\tepseq\right)\;, \ \
\Psi_j^{(s)}(H^{\anchor})=
\sum_{i\in S_j}H^{\anchor}\left(\bsX_{i-\dhinterseq,i+\dhinterseq}/\tepseq\right)
\;
\end{align*}
and
$$
W_n=\sum_{i=q_n+1}^{n}H^{\anchor}\left(\bsX_{i-\dhinterseq,i+\dhinterseq}/\tepseq\right)
=\sum_{i=n-r_n+1}^{n}H^{\anchor}\left(\bsX_{i-\dhinterseq,i+\dhinterseq}/\tepseq\right)\;.
$$

With such the decomposition, $\bsX_{1-\dhinterseq},\ldots,\bsX_{z_n\dhinterseq+\dhinterseq}$ used in the definition of $\Psi_1^{(l)}(H^{\anchor})$ are separated
by at least $\dhinterseq$ from the random variables that define $\Psi_2^{(l)}(H^{\anchor})$. The mixing condition \eqref{eq:mixing-rates-0} allows us to replace $\bsX$ with the independent blocks process, that is, we can treat the random variables $\Psi_j^{(l)}(H^{\anchor})$, $j=1,\ldots,\widetilde{m}_n$, as independent. The same applies to $\Psi_j^{(s)}(H^{\anchor})$.

Set
\begin{align}\label{eq:Zn-process}
{\mathbb{Z}}_n(H^{\anchor})=\sum_{j=1}^{\widetilde{m}_n}\left\{Z_{n,j}(H^{\anchor})-\esp[Z_{n,j}(H^{\anchor})]\right\}=:
\sum_{j=1}^{\widetilde{m}_n}\bar Z_{n,j}(H^{\anchor})
\end{align}
with
\begin{align}\label{eq:Zn-process-summands}
Z_{n,j}(H^{\anchor})=\frac{1}{\sqrt{\statinterseqn}}\Psi_j^{(l)}(H^{\anchor})\;.
\end{align}
The next steps are standard.
\begin{itemize}
	\item
	First, we show that the limiting variance of the large blocks process ${\mathbb{Z}_n}$ is the same as that of the process $\mathbb{F}_n$;
	\item Next, we show that the small blocks process (the scaled second term in
	\eqref{eq:large-small-block-process}) is negligible;
	\item We show that the boundary term $W_n$ is also negligible;
	\item Finally, we will verify the Lindeberg condition for the large blocks process.
\end{itemize}

\noindent {\bf Variance of the large blocks.}
We have (using the assumed independence of $\Psi_j^{(l)}(H^{\anchor})$)
\begin{align}\label{eq:variance-large-blocks}
&\var\left(\frac{1}{\sqrt{\statinterseqn}}\sum_{j=1}^{\widetilde{m}_n}\Psi_j^{(l)}(H^{\anchor})\right)
=\frac{ \widetilde{m}_n}{\statinterseqn}\var(\Psi_1^{(l)}(H^{\anchor}))\notag\\
&\sim\frac{1}{ z_n\dhinterseq \pr(\norm{\bsX_0}>\tepseq)}\var\left(\sum_{i=1}^{z_n\dhinterseq}H^{\anchor}\left(\bsX_{i-\dhinterseq,i+\dhinterseq}/\tepseq\right)\right)
\notag\\
&= \frac{1}{z_n\dhinterseq \pr(\norm{\bsX_0}>\tepseq)}\var\left(\sum_{j=0}^{z_n-1}H_{n,j}^{\anchor}\right)
\;,
\end{align}
where $H_{n,j}^\anchor$ is defined in \eqref{eq:Hnj} and
where in the last line we decomposed the block $L_1=\{1,\ldots,z_n\dhinterseq\}$ into $z_n$ disjoint blocks $J_0,\ldots,J_{z_n-1}$ and $\widetilde{m}_n\sim m_n/z_n$.
The next steps follow easily from \eqref{eq:covariance-calculation} with $m_n$ replaced by $z_n$.

The term in \eqref{eq:variance-large-blocks} becomes
\begin{align}\label{eq:decomposition-variance-large}
\frac{\var\left(H_{n,0}^{\anchor}\right)}{\dhinterseq\pr(\norm{\bsX_0}>\tepseq)}
+\frac{2}{\dhinterseq \pr(\norm{\bsX_0}>\tepseq)}\sum_{j=1}^{z_n-1}\left(1-\frac{j}{z_n}\right)
\cov(H_{n,0}^{\anchor},H_{n,j}^{\anchor})\;.
\end{align}
It follows immediately from \eqref{eq:asymptotics-1} that the limit of the first term above is
\begin{align}\label{eq:asymptotics-1new}
&\lim_{n\to\infty}
\frac{\var\left(H_{n,0}^{\anchor}\right)}{\dhinterseq\pr(\norm{\bsX_0}>\tepseq)}
=\tailmeasurestar(H^2)\;.
\end{align}
Now, for the second term in \eqref{eq:decomposition-variance-large} we adapt the proof of \Cref{lem:covariance-runs} from $m_n$ to $z_n$.

As in \eqref{eq:covariance-term},
for $j\geq 1$,
\begin{align*}
\frac{1}{\dhinterseq\pr(\norm{\bsX_0}>\tepseq)}\left|\cov(H_{n,0}^\anchor,{H}_{n,j}^{{\anchor}})\right|
&
\leq \sum_{h=(j-1)\dhinterseq+1}^{(j+1)\dhinterseq}|{g}_n(h)|=:I_{j}
\end{align*}
with (this time)
\begin{align*}
{g}_n(h)&=\frac{1}{\pr(\norm{\bsX_0}>\tepseq)}\cov(H^{\anchor}(\bsX_{-\dhinterseq,\dhinterseq}),
{H}^{\anchor}(\bsX_{h-\dhinterseq,h+\dhinterseq}))\;.
\end{align*}
For $h>2\dhinterseq$, similarly to \eqref{eq:covariance-term-1}, we have by \eqref{eq:davydov-2},
\begin{align*}
|{g}_n(h)|&\leq \frac{\|H\|_{\infty}^2}{\pr(\norm{\bsX_0}>\tepseq)} \beta_{h-2\dhinterseq}\;.
\end{align*}
Thus,
\begin{align*}
&\frac{1}{\dhinterseq \pr(\norm{\bsX_0}>\tepseq)}   \sum_{j=4}^{z_n-1}|\cov(H_{n,0}^{\anchor},{H}_{n,j}^{\anchor})|\leq
\frac{\|H\|_{\infty}^2}{\pr(\norm{\bsX_0}>\tepseq)}  \sum_{j=4}^{z_n-1}
\sum_{h=(j-1)\dhinterseq+1}^{(j+1)\dhinterseq}\beta_{h-2\dhinterseq}\\
&\leq 2\frac{\|H\|_{\infty}^2}{\pr(\norm{\bsX_0}>\tepseq)}  \sum_{h=3\dhinterseq+1}^{\infty}\beta_{h-2\dhinterseq}=O(1)\frac{1}{\pr(\norm{\bsX_0}>\tepseq)}
\sum_{i=\dhinterseq+1}^{\infty}\beta_{i}=o(1)
\end{align*}
by the assumption \eqref{eq:mixing-rates-runs}.
The terms that correspond to $j=1,2,3$ in \eqref{eq:covariance-calculation} are negligible.

In summary, we showed that
$$
\lim_{n\to\infty}\var\left(\frac{1}{\sqrt{\statinterseqn}}\sum_{j=1}^{\widetilde{m}_n}\Psi_j^{(l)}(H^{\anchor})\right)
=\tailmeasurestar(H^2)\;.
$$

\noindent {\bf Variance of the small blocks.} We have (using again the assumed independence of $\Psi_j^{(s)}(H^{\anchor})$ thanks to the beta-mixing)
\begin{align*}
\var\left(\frac{1}{\sqrt{\statinterseqn}}\sum_{j=1}^{\widetilde{m}_n}\Psi_j^{(s)}(H^{\anchor})\right)
=\frac{ \widetilde{m}_n}{\statinterseqn}\var(\Psi_1^{(s)}(H^{\anchor}))
\sim \frac{1}{ z_n\dhinterseq \pr(\norm{\bsX_0}>\tepseq)}\var(\Psi_1^{(s)}(H^{\anchor}))\;.
\end{align*}
Since the size of $\Psi_1^{(s)}(H^{\anchor})$ is $3$ times the size of $H_{n,1}^{\anchor}$ defined in \eqref{eq:Hnj}, we have by \eqref{eq:asymptotics-1}
\begin{align*}
\var\left(\frac{1}{\sqrt{\statinterseqn}}\sum_{j=1}^{\widetilde{m}_n}\Psi_j^{(s)}(H^{\anchor})\right)
\sim \frac{1}{z_n\dhinterseq \pr(\norm{\bsX_0}>\tepseq)}\dhinterseq \pr(\norm{\bsX_0}>\tepseq)\tailmeasurestar(H^2)
=O(1/z_n)=o(1)\;.
\end{align*}

\noindent {\bf Variance of the boundary term $W_n$.} We have (cf. \eqref{eq:Hnj})
\begin{align*}
\var\left(\frac{1}{\sqrt{\statinterseqn}}W_n\right)=
\var\left(\frac{1}{\sqrt{\statinterseqn}}
\sum_{i=1}^{\dhinterseq}H^{\anchor}\left(\bsX_{i-\dhinterseq,i+\dhinterseq}/\tepseq\right)\right)
=\frac{\var(H_{n,0}^\anchor)}{\statinterseqn}=\frac{\var(H_{n,0}^\anchor)}{n\pr(\norm{\bsX_0}>\tepseq)}\;.
\end{align*}
The latter term vanishes when $n\to\infty$, using
\eqref{eq:asymptotics-1new} and $\dhinterseq/n\to 0$.

\noindent {\bf Lindeberg condition for ${\mathbb Z}_n(H^{\anchor})$.}
We need to show that for all $\eta>0$,
\begin{align}\label{eq:lindeberg}
\lim_{n\to\infty}\widetilde{m}_n\esp\left[Z_{n,1}^2(H^{\anchor})\ind{|Z_{n,1}(H^{\anchor})|> \eta}\right]=0\;.
\end{align}
Since $H$ is bounded, then by \eqref{eq:zn},
\begin{align}\label{eq:lindeberg-1}
|Z_{n,1}(H^{\anchor})|\leq \frac{\sqrt{\statinterseqn}z_n\dhinterseq}{n \pr(\norm{\bsX_0}>\tepseq)}\|H\|_{\infty}\sim
\frac{z_n\dhinterseq}{\sqrt{n\pr(\norm{\bsX_0}>\tepseq)}}\|H\|_{\infty}=o(1)\;.
\end{align}
Thus, the indicator in \eqref{eq:lindeberg} becomes zero for large $n$.

\subsection{Proof of \Cref{thm:sliding-blocks-process-clt} - asymptotic equicontinuity}\label{sec:tightness}
We need the following lemma which is an adapted version of Theorem~2.11.1 in \cite{vandervaart:wellner:1996}.
Let $\mathbb{Z}_n$ be the empirical process indexed by a \semimetric\ space $(\mcg,\metricmcg)$, defined by
\begin{align*}
\bbZ_n(f) = \sum_{j=1}^{\widetilde{m}_n}\left\{Z_{n,j}(f)-\esp[Z_{n,j}(f)]\right\} \; ,
\end{align*}
where $\{Z_{n,j},n\geq 1\}$, $j=1,\ldots,\widetilde{m}_n$, are \iid\ separable, stochastic processes and $\widetilde{m}_n$
is a sequence of integers such that $\widetilde{m}_n\to\infty$.
Define the random \semimetric\ $d_n$ on $\mcg$ by
\begin{align*}
d_n^2(f,g) = \sum_{j=1}^{\widetilde{m}_n} \{Z_{n,j}(f)-Z_{n,j}(g)\}^2 \; , f,g\in \mcg \; .
\end{align*}
\begin{lemma}
	\label{theo:VW2.11.1}
	Assume that $(\mcg,\metricmcg)$ is totally bounded.  Assume moreover that:
	\begin{enumerate}[(i)]
		\item \label{item:lindeberg-envelope}  For all $\eta>0$,
		\begin{align}
		\label{eq:lindeberg-envelope}
		\lim_{n\to\infty} {\widetilde{m}_n} \esp[\|Z_{n,1}\|^2_\mcg\ind{\|Z_{n,1}\|^2_\mcg>\eta}] = 0 \; .
		\end{align}
		\item \label{item:continuity-l2} For every sequence $\{\delta_n\}$ which decreases to zero,
		\begin{align}\label{eq:continuity-l2}
		\lim_{n\to\infty} \sup_{f,g\in\mcg\atop \rho(f,g)\leq\delta_n} \esp[d_n^2(f,g)] = 0 \; . 
		\end{align}
		\item \label{item:random-entropy} There exists a measurable majorant $N^*(\mcg,d_n,\epsilon)$
		of the covering number $N(\mcg,d_n,\epsilon)$ such that for every sequence $\{\delta_n\}$ which
		decreases to zero,
		\begin{align*}
		\int_0^{\delta_n} \sqrt{\log N^*(\mcg,d_n,\epsilon)} \rmd\epsilon\convprob 0  \; .
		\end{align*}
	\end{enumerate}
	Then $\{\mathbb{Z}_n,n\geq 1\}$ is asymptotically $\rho$-equicontinuous, \ie\ for each $\eta>0$,
	\begin{align*}
	\lim_{\delta\to0} \limsup_{n\to\infty} \pr\left(\sup_{{f,g\in \mcg}\atop{\rho(f,g)<\delta}} |\mathbb{Z}_n(f)-\mathbb{Z}_n(g)| > \eta\right) = 0 \; .
	\end{align*}
\end{lemma}
\begin{remark}
	The
	separability assumption is not in \cite{vandervaart:wellner:1996}.  It implies measurability of
	$\|Z_{n,1}\|_{\mcg}$.  Furthermore, the separability also implies that for all $\delta>0$,
	$n\in\Nset$, $(e_j)_{1\leq j\leq \widetilde{m}_n}\in\{-1,0,1\}^{\widetilde{m}_n}$ and $i\in \{1,2\}$, the supremum
	\begin{align*}
	\sup_{f,g\in\mcg \atop \metricmcg(f,g) <\delta} \left|\sum_{j=1}^{\widetilde{m}_n} e_j
	\left(Z_{n,j}(f)-Z_{n,j}(g)\right)^i\right|
	& = \sup_{f,g\in\mcg_0 \atop \metricmcg(f,g) <\delta} \left|\sum_{j=1}^{\widetilde{m}_n} e_j \left(Z_{n,j}(f)-Z_{n,j}(g)\right)^i\right| \;
	\end{align*}
	is measurable, which is an assumption of  \cite{vandervaart:wellner:1996}.
	\remarkend
\end{remark}

\subsubsection{Asymptotic equicontinuity of the empirical process of sliding blocks}\label{sec:tightness-1}
Recall the big-blocks process ${\mathbb{Z}}_n(H^{\anchor})$ (cf. \eqref{eq:Zn-process}-\eqref{eq:Zn-process-summands}).
Recall also that thanks to the $\beta$-mixing we can consider random variables $\Psi_j^{(l)}(H^{\anchor})$, $j=1,\ldots,\widetilde{m}_n$ to be independent.
Recall that $H^{\anchor}_s$ is defined in \eqref{eq:HC with s}.
We need to prove the asymptotic equicontinuity of ${\mathbb{Z}}_n(H_s^{\anchor})$
indexed by the class $\mcg = \{H^{\anchor}_s,s\in[s_0,t_0]\}$ equipped with the metric $\ltwotmsmetric(H_s^{\anchor},H_t^\anchor)=\tailmeasurestar(\{H_s^{\anchor}-{H}_t^{\anchor}\}^2)$.
The same argument can be used to prove the asymptotic equicontinuity for the small blocks process. This yields asymptotic equicontinuity of
${\mathbb{F}}_n(H^{\anchor}_{s})$.

In what follows, the proof of the Lindeberg-type condition \eqref{eq:lindeberg-envelope} is easy. The proof of
\eqref{eq:continuity-l2} is quite involved.

Thanks to \Cref{hypo:entropy}, the condition \eqref{eq:randomentropy}
is satisfied. Its validity is discussed in \Cref{sec:random-entropy}.

\paragraph{Lindeberg condition: Proof of \eqref{eq:lindeberg-envelope}.}
We re-write \eqref{eq:lindeberg-1} as follows:
\begin{align*}
\sup_{s\in [s_0,t_0]}|Z_{n,1}(H_s^{\anchor})|\leq \frac{\sqrt{\statinterseqn}z_n\dhinterseq}{n \pr(\norm{\bsX_0}>\tepseq)}\|H\|_{\infty}\leq
\frac{z_n\dhinterseq}{\sqrt{n\pr(\norm{\bsX_0}>\tepseq)}}\sup_{s\in [s_0,t_0]}\|H_s\|_{\infty}\;,
\end{align*}
Since the class $\{H_s:s\in [s_0,t_0]\}$ is linearly ordered, $\sup_{s\in [s_0,t_0]}\|H_s\|_{\infty}$ is achieved either at $s=s_0$ or $s=t_0$. Hence,
the Lindeberg condition~\ref{item:lindeberg-envelope} of~\Cref{theo:VW2.11.1} holds by \eqref{eq:lindeberg}.

\paragraph{Asymptotic continuity of random semi-metric:
	Proof of \eqref{eq:continuity-l2}.}

The proof is rather long and technical.

Define the random metric
\begin{align*}
d_n^2(H_s^{\anchor},{H}_t^{\anchor}) =
\sum_{j=1}^{\widetilde m_n} (Z_{n,j}(H_s^{\anchor})-Z_{n,j}({H}_t^{\anchor}))^2 \; .
\end{align*}
Let (cf. \eqref{eq:Hnj})
\begin{align*}
H_{s,n,j}^{\anchor}&=\sum_{i=j\dhinterseq+1}^{(j+1)\dhinterseq}
H\left(\bsX_{i-\dhinterseq,i+\dhinterseq}/(s\tepseq)\right)
\ind{\anchor\left(\bsX_{i-\dhinterseq,i+\dhinterseq}/(s\tepseq)\right)=i}\ind{\norm{\bsX_i}>s\tepseq}\;.
\end{align*}
We need to evaluate
$\esp[d_n^2(H^{\anchor}_s,H^{\anchor}_t)]$:
\begin{align}\label{eq:tightness-evaluation-1}
&\esp[d_n^2(H^{\anchor}_s,H^{\anchor}_t)]\notag\\
&
=\frac{\statinterseqn\widetilde{m}_n}{(n\pr(\norm{\bsX_0}>\tepseq))^2}
\esp\left[\left(\sum_{i=1}^{z_n\dhinterseq}\left\{H^{\anchor}_s\left(\bsX_{i-\dhinterseq,i+\dhinterseq}/\tepseq\right)-
H^{\anchor}_t\left(\bsX_{i-\dhinterseq,i+\dhinterseq}/\tepseq\right)\right\}\right)^2\right]
\notag\\
&\sim \frac{1}{ z_n\dhinterseq \pr(\norm{\bsX_0}>\tepseq)}
\esp\left[\left(\sum_{j=0}^{z_n-1}\left\{H_{s,n,j}^{\anchor}-H_{t,n,j}^{\anchor}\right\}\right)^2\right]
\;,
\end{align}
where in the last line we decomposed the block $L_1$ into $z_n$ disjoint blocks $J_0,\ldots,J_{z_n-1}$,  $\widetilde{m}_n\sim m_n/z_n$; cf. \eqref{eq:variance-large-blocks}.
The term in \eqref{eq:tightness-evaluation-1} becomes
\begin{align} &\frac{\esp[\left\{H_{s,n,0}^{\anchor}-H_{t,n,0}^{\anchor}\right\}^2]}{\dhinterseq\pr(\norm{\bsX_0}>\tepseq)}\nonumber\\
&+2\frac{1}{\dhinterseq \pr(\norm{\bsX_0}>\tepseq)}\sum_{j=1}^{z_n-1}\left(1-\frac{j}{z_n}\right) \esp[\left\{H_{s,n,0}^{\anchor}-H_{t,n,0}^{\anchor}\right\}\left\{H_{s,n,j}^{\anchor}-H_{t,n,j}^{\anchor}\right\}]
\nonumber
\end{align}
The above lines correspond  to \eqref{eq:covariance-calculation} with $m_n$ replaced by $z_n$.

We are going to prove two statements:
\begin{align}\label{eq:bound-on-the-first-term-case-1}
\lim_{n\to\infty}\sup_{s_0\leq s, t \leq t_0 \atop |s-t|\leq\delta_n} \frac{\esp[\left\{H_{s,n,0}^{\anchor}-H_{t,n,0}^{\anchor}\right\}^2]}{\dhinterseq\pr(\norm{\bsX_0}>\tepseq)}=0\;
\end{align}
and
\begin{align}\label{eq:bound-on-the-second-term-case-1}
\lim_{n\to\infty}\sup_{s_0\leq s, t \leq t_0 \atop |s-t|\leq\delta_n}
\frac{1}{\dhinterseq \pr(\norm{\bsX_0}>\tepseq)}\sum_{j=1}^{z_n-1}\left(1-\frac{j}{z_n}\right) \esp[\left\{H_{s,n,0}^{\anchor}-H_{t,n,0}^{\anchor}\right\}\left\{H_{s,n,j}^{\anchor}-H_{t,n,j}^{\anchor}\right\}]=0\;.
\end{align}

\paragraph{Proof of \eqref{eq:bound-on-the-first-term-case-1}.}
We will write $\{H_s^\anchor-H_t^\anchor\}(\bsx)$ for
$H_s^\anchor(\bsx)-H_t^\anchor(\bsx)$.

Similarly to \eqref{eq:asymptotics-1},
\begin{align}
&\frac{\esp[\left\{H_{s,n,0}^{\anchor}-H_{t,n,0}^{\anchor}\right\}^2]}{\dhinterseq\pr(\norm{\bsX_0}>\tepseq)}\nonumber\\
&\leq
\|H\|\frac{1}{\pr(\norm{\bsX_0}>\tepseq)}\sum_{h=-\dhinterseq}^{\dhinterseq}
\left|\esp[
\{H_s^{\anchor}-H_t^\anchor\}\left(\bsX_{-\dhinterseq,\dhinterseq}/\tepseq\right)\times
\{{H}_s^{\anchor}-H_t^\anchor\}\left(\bsX_{h-\dhinterseq,h+\dhinterseq}/\tepseq\right)]\right|\nonumber\\
&=:\|H\|\sum_{h=-\dhinterseq}^{\dhinterseq}|g_n(h,H_s^\anchor-H_t^\anchor)|
\;\label{eq:decomposition-0-term-tightness}
\end{align}
with
\begin{align}\label{eq:bound-gn-E}
|{g}_n(h,G)|=\left|\frac{1}{\pr(\norm{\bsX_0}>\tepseq)}\esp[G(\bsX_{-\dhinterseq,\dhinterseq}/\tepseq)
G(\bsX_{h-\dhinterseq,h+\dhinterseq}/\tepseq)]\right|\;.
\end{align}
Using the definition \eqref{eq:HC with s} of $H_s^\anchor$, the fact that $s,t\geq s_0$ and since $H$ is bounded, we immediately get
\begin{align}\label{eq:immediate-bound-on-gn}
|{g}_n(h,H_s^\anchor-H_t^\anchor)|\leq \frac{4}{\pr(\norm{\bsX_0}>\tepseq)}\|H\|^2 \pr(\norm{\bsX_0}>s_0\tepseq,\norm{\bsX_h}>s_0\tepseq)\;.
\end{align}
To get a more precise bound that involves the difference $s-t$ we need to consider three cases. The reason for this is that we need to keep the absolute value in \eqref{eq:bound-gn-E} outside of the expectation. As such, computations below are quite technically involved.

To shorten our displays, we introduce the notation
\begin{align}\label{eq:indicator-notation}
\mci(i,s):=\ind{\norm{\bsX_i}>s\tepseq}\;.
\end{align}
\paragraph{Case 1.} Assume here that $\anchor$ is $0$-homogeneous. Then for any $i$,
\begin{align}\label{eq:specification-of-Hs}
H_s^\anchor\left(\bsX_{i-\dhinterseq,i+\dhinterseq}/\tepseq\right)&=
H\left(\bsX_{i-\dhinterseq,i+\dhinterseq}/(s\tepseq)\right)
\ind{\anchor\left(\bsX_{i-\dhinterseq,i+\dhinterseq}/(s\tepseq)\right)=i}\mci(i,s)\\
&=H_s\left(\bsX_{i-\dhinterseq,i+\dhinterseq}/\tepseq\right)
\ind{\anchor\left(\bsX_{i-\dhinterseq,i+\dhinterseq}/\tepseq\right)=i}\mci(i,s)\;;
\nonumber
\end{align}
(we keep $\tepseq$ in the argument of $\anchor$, although it can be omitted). What is important in this decomposition is that we can control monotonicity (withe respect to $s$) of each term.

Then
\begin{align*}
&H_s^\anchor\left(\bsX_{i-\dhinterseq,i+\dhinterseq}/\tepseq\right)-H_t^\anchor\left(\bsX_{i-\dhinterseq,i+\dhinterseq}/\tepseq\right)=\\
&=\ind{\anchor\left(\bsX_{i-\dhinterseq,i+\dhinterseq}/\tepseq\right)=i}
H_s\left(\bsX_{i-\dhinterseq,i+\dhinterseq}/\tepseq\right)\Big(\mci(i,s)-\mci(i,t)\Big)\\
&\phantom{=}+
\ind{\anchor\left(\bsX_{i-\dhinterseq,i+\dhinterseq}/\tepseq\right)=i}
\mci(i,t)\Big(H_s\left(\bsX_{i-\dhinterseq,i+\dhinterseq}/\tepseq\right)-
H_t\left(\bsX_{i-\dhinterseq,i+\dhinterseq}/\tepseq\right)\Big)\\
&=: T_1(i)+T_2(i)
\end{align*}
and hence
\begin{align*}
\left|\esp\left[\{H_s^{\anchor}-H_t^\anchor\}\left(\bsX_{-\dhinterseq,\dhinterseq}/\tepseq\right) \{H_s^{\anchor}-H_t^\anchor\}\left(\bsX_{h-\dhinterseq,h+\dhinterseq}/\tepseq\right)\right]\right|
\end{align*}
is bounded by the sum of four nonnegative terms $W_{11}+W_{22}+W_{12}+W_{21}$ that we are going to define below. The general idea is that we will obtain rough bounds in terms of the difference
$\Big(\mci(i,s)-\mci(i,t)\Big)$, except of one case which involves the product $T_2(0)T_2(h)$.
Coming back to the definitions of $W_{i,j}$, the double sub-index $\{12\}$ of $W$ indicates that $W_{12}$ is related to multiplying $T_1(0)$ by $T_2(h)$; $W_{21}$ means we multiply $T_2(0)$ and $T_1(h)$):
\begin{align*}
W_{11}:=\|H\|^2\esp\left[\left(\mci(0,s)-\mci(0,t)\right)
\left(\mci(h,s)-\mci(h,t)\right)\right]\;,
\end{align*}
(above, the indicators of the anchoring map are omitted);
\begin{align*}
&W_{22}:=\\
&\esp\left[\ind{\anchor\left(\bsX_{-\dhinterseq,\dhinterseq}/\tepseq\right)=0}\mci(0,t)\left\{H_s-H_t\right\}\left(\bsX_{-\dhinterseq,\dhinterseq}/\tepseq\right)
\left\{H_s-H_t\right\}\left(\bsX_{h-\dhinterseq,h+\dhinterseq}/\tepseq\right)\right]\;,
\end{align*}
(above, $\ind{\anchor\left(\bsX_{h-\dhinterseq,h+\dhinterseq}/\tepseq\right)=h}$ and $\mci(h,t)$ are omitted);
\begin{align*}
&W_{12}:=\pm \|H\|\esp\left[\Big(\mci(0,s)-\mci(0,t)\Big)
\left\{H_s-H_t\right\}\left(\bsX_{h-\dhinterseq,h+\dhinterseq}/\tepseq\right)\right]\;,
\end{align*}
(both indicators of the anchoring maps and $\mci(h,t)$ are omitted);
\begin{align*}
&W_{21}:=\pm \|H\|\esp\left[\Big(\mci(h,s)-\mci(h,t)\Big)
\left\{H_s-H_t\right\}\left(\bsX_{-\dhinterseq,\dhinterseq}/\tepseq\right)\right]\;,
\end{align*}
(both indicators of the anchoring maps and $\mci(0,t)$ are omitted).

Note that the \rhs\ of both $W_{11}, W_{22}$ is nonnegative thanks to the monotonicity and $s<t$, while for the \rhs\ of $W_{12},W_{21}$ we need to put $\pm$, since the sign of the expressions there depends on whether the map $s\to H_s$ is decreasing or increasing.

Thus,
\begin{align}\label{eq:W1}
W_{11}\leq \|H\|^2 \left(\pr(\norm{\bsX_0}>\tepseq s)-\pr(\norm{\bsX_0}>\tepseq t)\right)\;,
\end{align}
\begin{align}\label{eq:W2}
&W_{22}\leq \pm 2\|H\|\esp\left[\ind{\anchor\left(\bsX_{-\dhinterseq,\dhinterseq}/\tepseq\right)=0}\mci(0,t)\left\{H_s-H_t\right\}\left(\bsX_{-\dhinterseq,\dhinterseq}/\tepseq\right)\right]\;,
\end{align}
\begin{align}\label{eq:W3}
&W_{12}+W_{21}\leq 4\|H\|^2 \left(\pr(\norm{\bsX_0}>\tepseq s)-\pr(\norm{\bsX_0}>\tepseq t)\right)\;.
\end{align}
The bound on $W_{22}$ (with $+$) is obvious if $H_s-H_t\geq 0$ (thus, $s\to H_s$ is decreasing), while in an increasing case of $s\to H_s$ we use the following observation: if $a,b<0$, $|b|<c$, then $ab\leq -ac$ (yielding $-$ on the \rhs\ of \eqref{eq:W2}).   The bound on $W_{12}+W_{21}$ follows from the same reasoning.

In summary, with $g_n(h,H_s^\anchor-H_t^\anchor)$ defined in \eqref{eq:bound-gn-E}, we have
\begin{align}
&|g_n(h,H_s^\anchor-H_t^\anchor)|\leq 5\|H\|^2 \frac{\pr(\norm{\bsX_0}>\tepseq s)-\pr(\norm{\bsX_0}>\tepseq t)}{\pr(\norm{\bsX_0}>\tepseq)}\label{eq:tightness-0}\\
&\pm 2\|H\|\frac{\esp\left[\ind{\anchor\left(\bsX_{-\dhinterseq,\dhinterseq}/\tepseq\right)=0}\ind{\norm{\bsX_0}>s_0\tepseq}\left\{H_s-H_t\right\}\left(\bsX_{-\dhinterseq,\dhinterseq}/\tepseq\right)\right]}{\pr(\norm{\bsX_0}>\tepseq)}\;,
\label{eq:tightness-1}
\end{align}
where again the presence of $\pm$ depends on the sign of $H_s-H_t$.

We can ignore the scaling factor $2\|H\|$ in \eqref{eq:tightness-1} and write it as (recall that the anchoring map is $0$-homogeneous)
\begin{align*}
&\esp[H_{s/s_0}\left(\bsX_{-\dhinterseq,\dhinterseq} /(s_0\tepseq)\right)
\ind{\anchor\left(\bsX_{-\dhinterseq,\dhinterseq} /(s_0\tepseq)\right)=0}\mid \norm{\bsX_0}>s_0\tepseq]	\frac{\pr(\norm{\bsX_0}>s_0\tepseq)}{\pr(\norm{\bsX_0}>\tepseq)}\\
&\phantom{=}-\esp[H_{t/s_0}\left(\bsX_{-\dhinterseq,\dhinterseq} /(s_0\tepseq)\right)
\ind{\anchor\left(\bsX_{-\dhinterseq,\dhinterseq} /(s_0\tepseq)\right)=0}\mid \norm{\bsX_0}>s_0\tepseq]\frac{\pr(\norm{\bsX_0}>s_0\tepseq)}{\pr(\norm{\bsX_0}>\tepseq)}\\
&= \Big( \mu_{n,\dhinterseq}(s)  - \mu_{n,\dhinterseq}(t)\Big)\frac{\pr(\norm{\bsX_0}>s_0\tepseq)}{\pr(\norm{\bsX_0}>\tepseq)}=:
\Big( \tilde\mu_{n,\dhinterseq}(s)  - \tilde\mu_{n,\dhinterseq}(t)\Big)\;,
\end{align*}
with
\begin{align*}
\mu_{n,\dhinterseq}(\cdot )&=\esp[H_{\cdot/s_0}\left(\bsX_{-\dhinterseq,\dhinterseq} /(s_0 \tepseq)\right)
\ind{\anchor\left(\bsX_{-\dhinterseq,\dhinterseq} /(s_0 \tepseq)\right)=0}\mid \norm{\bsX_0}>s_0 \tepseq]\;
\end{align*}
and
$$
\tilde\mu_{n,\dhinterseq}(s)=\mu_{n,\dhinterseq}(s)\frac{\pr(\norm{\bsX_0}>s_0\tepseq)}{\pr(\norm{\bsX_0}>\tepseq)}\;.
$$
Thanks to \eqref{eq:1dim-conv}, $\lim_{n\to\infty}\mu_{n,\dhinterseq}(s)=\tailmeasurestar(H_{s/s_0})$. Thanks to the monotonicity of $s\to H_S$ and homogeneity of $\tailmeasurestar$,
the convergence of $\tilde\mu_{n,\dhinterseq}(s)$
to $$s_0^{-\alpha}\tailmeasurestar(H_{s/s_0})=s_0^{-\alpha}(s/s_0)^{-\alpha}\tailmeasurestar(H)=s^{-\alpha}\tailmeasurestar(H)$$ is uniform on $[s_0,t_0]$. Thus,
for $s,t\in[s_0,t_0]$,
\begin{multline*}
\left| \tilde\mu_{n,\dhinterseq}(s)  - \tilde\mu_{n,\dhinterseq}(t) \right|
\leq 2 \sup_{s_0\leq u \leq t_0} \left| \tilde\mu_{n,\dhinterseq}(u ) -
\tailmeasurestar(H_{u}) \right| + \tailmeasurestar(H) \{s^{-\alpha} - t^{-\alpha}\}\;.
\end{multline*}
Fix $\eta>0$. For large enough $n$, the uniform convergence yields
\begin{multline}
\sup_{s_0\leq s, t \leq t_0 \atop |s-t|\leq\delta_n} \left|
\tilde\mu_{n,\dhinterseq}(s) - \tilde\mu_{n,\dhinterseq}(t) \right|
\leq \eta + \tailmeasurestar(H) \sup_{s_0\leq s, t \leq t_0 \atop |s-t|\leq\delta_n}  \{s^{-\alpha} - t^{-\alpha}\} \\
\leq \eta + \alpha s_0^{-\alpha-1} \delta_n \tailmeasurestar(H)\; .\label{eq:bounds for g(h)}
\end{multline}
The uniform convergence also yields that the term in \eqref{eq:tightness-0} is bounded by $\eta + \alpha s_0^{-\alpha-1} \delta_n$. This, together with \eqref{eq:bounds for g(h)}, gives
\begin{align}\label{eq:s-t-bound-on-gn}
|g_n(h,H_s^\anchor-H_t^\anchor)|\leq \eta + \constant\ \delta_n
\end{align}
with a generic constant $\constant$.

Fix an integer $r$. Using \eqref{eq:immediate-bound-on-gn} and \eqref{eq:s-t-bound-on-gn}
we have
\begin{align*}
&\sup_{s_0\leq s, t \leq t_0 \atop |s-t|\leq\delta_n}\frac{\esp[\left\{H_{s,n,0}^{\anchor}-H_{t,n,0}^{\anchor}\right\}^2]}{\dhinterseq\pr(\norm{\bsX_0}>\tepseq)}\\
&\leq \sup_{s_0\leq s, t \leq t_0 \atop |s-t|\leq\delta_n}\sum_{h=-r}^{r}|g_n(h,H_s^\anchor-H_t^\anchor)|+
\sup_{s_0\leq s, t \leq t_0 \atop |s-t|\leq\delta_n}\sum_{|h|=r}^{\dhinterseq}|g_n(h,H_s^\anchor-H_t^\anchor)|\\
&
\leq \constant\ r(\eta + \delta_n)+\frac{4}{\pr(\norm{\bsX_0}>\tepseq)}\sum_{|h|=r}^{\dhinterseq} \pr(\norm{\bsX_0}>s_0\tepseq,\norm{\bsX_h}>s_0\tepseq)\;.
\end{align*}
Applying the anticlustering conditions \ref{eq:conditionS} to the second term, letting $\delta_n\to 0$, since $\eta$ is arbitrary, this proves \eqref{eq:bound-on-the-first-term-case-1}.

\paragraph{Case 2.}
Now, we consider the anchoring maps $\anchor^{(1)}$ and $\anchor^{(2)}$ which are not 0-homogeneous.
Note that for $j=1,2$ we can write
\begin{align*}
H_s^{\anchor^{(j)}}\left(\bsX_{i-\dhinterseq,i+\dhinterseq}/\tepseq\right)&=
H\left(\bsX_{i-\dhinterseq,i+\dhinterseq}/(s\tepseq)\right)
\ind{\anchor^{(j)}\left(\bsX_{i-\dhinterseq,i+\dhinterseq}/(s\tepseq)\right)=i}\mci(i,s)\\
&= H_s\left(\bsX_{i-\dhinterseq,i+\dhinterseq}/\tepseq\right) \mci(i,s)
F_s\left(\bsX_{i-\dhinterseq,i+\dhinterseq}/\tepseq\right)\;,
\end{align*}
where
\begin{align*}
F_s\left(\bsX_{i-\dhinterseq,i+\dhinterseq}/\tepseq\right)=
\ind{\bsX_{i-\dhinterseq,i-1}^*\leq s\tepseq}
\end{align*}
or
\begin{align*}
F_s\left(\bsX_{i-\dhinterseq,i+\dhinterseq}/\tepseq\right)=
\ind{\bsX_{i+1,i+\dhinterseq}^*\leq s\tepseq}
\end{align*}
in case $j=1$ and $j=2$, respectively. Note that regardless of the monotonicity of the map $s\to \anchor_s^{(j)}$, the map $s\to F_s$ is always \underline{non-decreasing}.
Then \eqref{eq:specification-of-Hs} gives, for any $i\in\Nset$, and $\anchor=\anchor^{(1)},\anchor^{(2)}$,
\begin{align*}
&H_s^\anchor\left(\bsX_{i-\dhinterseq,i+\dhinterseq}/\tepseq\right)-H_t^\anchor\left(\bsX_{i-\dhinterseq,i+\dhinterseq}/\tepseq\right)=\\
&=
H_s\left(\bsX_{i-\dhinterseq,i+\dhinterseq}/\tepseq\right)
F_s\left(\bsX_{i-\dhinterseq,i+\dhinterseq}/\tepseq\right)\bigg(\mci(i,s)-\mci(i,t)\bigg)
\\
&\phantom{=}+\mci(i,t)F_s\left(\bsX_{i-\dhinterseq,i+\dhinterseq}/\tepseq\right)\big(H_s\left(\bsX_{i-\dhinterseq,i+\dhinterseq}/\tepseq\right)-H_t\left(\bsX_{i-\dhinterseq,i+\dhinterseq}/\tepseq\right)\big)\\
&\phantom{=}+\mci(i,t)H_t\left(\bsX_{i-\dhinterseq,i+\dhinterseq}/\tepseq\right)
\bigg(
F_s\left(\bsX_{i-\dhinterseq,i+\dhinterseq}/\tepseq\right)
-F_t\left(\bsX_{i-\dhinterseq,i+\dhinterseq}/\tepseq\right)\bigg)\\
&=: T_1(i)+T_2(i)+T_3(i)
\;.
\end{align*}
Again, in this decomposition we can control monotonicity of each term. The argument now is very similar to that of Case 1. Hence, it is omitted.

Therefore, \eqref{eq:bound-on-the-first-term-case-1} is proved for $\anchor^{(j)}$, $j=0,1,2$.
\paragraph{Proof of \eqref{eq:bound-on-the-second-term-case-1}.}
We proceed similarly.
Recall that for $j\geq 1$ (cf. \eqref{eq:covariance-term}),
\begin{align}
\label{eq:tightness-calculation-second-term}
&\frac{1}{\dhinterseq \pr(\norm{\bsX_0}>\tepseq)}\esp[\left\{H_{s,n,0}^{\anchor}-H_{t,n,0}^{\anchor}\right\}
\left\{H_{s,n,j}^{\anchor}-H_{t,n,j}^{\anchor}\right\}]\nonumber\\
&
=
\frac{1}{\dhinterseq\pr(\norm{\bsX_0}>\tepseq)}
\esp\left[
\sum_{h=1}^{\dhinterseq}\Big\{ H_s^\anchor-H_t^\anchor\Big\}\left(\bsX_{h-\dhinterseq,h+\dhinterseq}/\tepseq\right)
\times
\sum_{i=j\dhinterseq+1}^{(j+1)\dhinterseq}\Big\{ H_s^\anchor-H_t^\anchor\Big\}\left(\bsX_{i-\dhinterseq,i+\dhinterseq}/\tepseq\right)
\right]
\nonumber
\\
&=\sum_{h=(j-1)\dhinterseq+1}^{j\dhinterseq}\left(\frac{h}{\dhinterseq}-(j-1)\right){g}_n(h,H^{\anchor}_s-H^{\anchor}_t) +
\sum_{h=j\dhinterseq+1}^{(j+1)\dhinterseq}\left((j+1)-\frac{h}{\dhinterseq}\right){g}_n(h,H^{\anchor}_s-H^{\anchor}_t)\nonumber \nonumber
\\
&\leq \sum_{h=(j-1)\dhinterseq+1}^{(j+1)\dhinterseq}{g}_n(h,H^{\anchor}_s-H^{\anchor}_t)=:I_{j}(s,t)
\end{align}
with the same $g_n$ as in
\eqref{eq:bound-gn-E}.

Write ${g}_n(h,G)$ as
\begin{align*}
&\frac{1}{\pr(\norm{\bsX_0}>\tepseq)}\cov[G(\bsX_{-\dhinterseq,\dhinterseq}/\tepseq),
G(\bsX_{h-\dhinterseq,h+\dhinterseq}/\tepseq)]+
\frac{1}{\pr(\norm{\bsX_0}>\tepseq)}\esp^2[G(\bsX_{-\dhinterseq,\dhinterseq}/\tepseq)
]\\
&=:\widetilde{g}_n(h,G)+\frac{1}{\pr(\norm{\bsX_0}>\tepseq)}\esp^2[G(\bsX_{-\dhinterseq,\dhinterseq}/\tepseq)
]
\;.
\end{align*}

For $h>2\dhinterseq$ we have by \eqref{eq:davydov-2}, 
\begin{align}
|\widetilde{g}_n(h,H^{\anchor}_s-H^{\anchor}_t)|&\leq \frac{\constant}{\pr(\norm{\bsX_0}>\tepseq)} \beta_{h-2\dhinterseq}\;.
\end{align}
Thus,
\begin{align}\label{eq:bound-on-the-second-term-tail}
&\frac{1}{\dhinterseq \pr(\norm{\bsX_0}>\tepseq)}\sum_{j=4}^{z_n-1}\left(1-\frac{j}{z_n}\right)\esp[\left(H_{s,n,0}^{\anchor}-H_{t,n,0}^{\anchor}\right)\left(H_{s,n,j}^{\anchor}-H_{t,n,j}^{\anchor}\right)]\nonumber\\
&\leq
\frac{\constant}{\pr(\norm{\bsX_0}>\tepseq)}  \sum_{j=4}^{z_n-1}
\sum_{h=(j-1)\dhinterseq+1}^{(j+1)\dhinterseq}\beta_{h-2\dhinterseq}\nonumber \\
&\phantom{\leq}
+\frac{\constant}{\pr(\norm{\bsX_0}>\tepseq)}  z_n\dhinterseq \esp^2[\{H_s^\anchor-H_t^\anchor\}(\bsX_{-\dhinterseq,\dhinterseq}/\tepseq)]
\nonumber\\
&\leq \frac{\constant}{\pr(\norm{\bsX_0}>\tepseq)}  \sum_{h=3\dhinterseq+1}^{\infty}\beta_{h-2\dhinterseq}+\constant
z_n\dhinterseq \pr(\norm{\bsX_0}>s_0 \tepseq)
\nonumber\\
&=\frac{\constant}{\pr(\norm{\bsX_0}>\tepseq)}
\sum_{i=\dhinterseq+1}^{\infty}\beta_{i}+o(1)=o(1)
\end{align}
uniformly in $s,t\in [s_0,t_0]$.
In the last line we applied the assumption \eqref{eq:mixing-rates-runs}, and the assumption \eqref{eq:zn}.

The terms that correspond to $j=1,2,3$ in \eqref{eq:tightness-calculation-second-term} have to be dealt with separately. We note that $I_1(s,t)=\sum_{h=1}^{\dhinterseq} g_n(h,H_s^\anchor-H_t^\anchor)$ is bounded by the term in
\eqref{eq:decomposition-0-term-tightness}. Hence,
\begin{align}\label{eq:bound-on-the-second-term-1}
\lim_{n\to\infty}\sup_{s_0\leq s, t \leq t_0 \atop |s-t|\leq\delta_n} I_1(s,t)=0\;.
\end{align}
Next, using \eqref{eq:immediate-bound-on-gn} and \ref{eq:conditionS},
\begin{align}\label{eq:bound-on-the-second-term-2}
I_2(s,t)+I_3(s,t)\leq \constant \sum_{h=\dhinterseq+1}^{4\dhinterseq}g_n(h)\leq \constant \sum_{h=\dhinterseq+1}^{4\dhinterseq}
\pr(\norm{\bsX_0}>s_0\tepseq,\norm{\bsX_h}>s_0\tepseq)=0\;,
\end{align}
uniformly in $s,t\in [s_0,t_0]$.

Combination of \eqref{eq:bound-on-the-second-term-tail}, \eqref{eq:bound-on-the-second-term-1}, \eqref{eq:bound-on-the-second-term-2} finishes the proof of
\eqref{eq:bound-on-the-second-term-case-1}.

This, together with \eqref{eq:bound-on-the-first-term-case-1} concluded the proof of
\eqref{eq:continuity-l2}.

\subsection{Random entropy}\label{sec:random-entropy}
In this section we discuss validity of \Cref{hypo:entropy}. We cannot check this condition for arbitrary functionals $H$ and anchoring maps $\anchor$, however, we will see that the conditions is satisfied for most relevant cases considered in the paper.

Recall the class
$$\mcg = \{H^{\anchor}_s,s\in[s_0,t_0]\}=\{H(\bsx/s)\ind{\anchor(\bsx/s)=0}\ind{\norm{\bsx_0}>s},\;\;s\in[s_0,t_0]\}.$$ We start first with $H$ of the form
\begin{align}\label{eq:class-1}
H_{s}(\bsx)=\ind{K(\bsx)>s}\;,
\end{align}
where
$K:\Rset^\Zset\to \Rset$. This is the case of the functionals that lead to the extremal index, the large deviation index and the ruin index.

Since $\anchor^{(0)}$ is 0-homogeneous, it does not play a role in calculating the class entropy. Then
$$
H_{s}(\bsx)\ind{\norm{\bsx_0}>s}= \ind{\min\{K(\bsx),\norm{\bsx_0}\}>s}\;.
$$
Hence, the map $s\to H_{s}(\bsx)\ind{\norm{\bsx_0}>s}$ is decreasing. Therefore, VC($\mcg)=2$.

As for $\anchor^{(1)}$ we have
\begin{align*}
\ind{\anchor^{(1)}(\bsx/s)=0}=
\ind{ \bsx_{-\infty,-1}^*\leq s,\norm{\bsx_0}>s}\;.
\end{align*}
Thus,
\begin{align*}
H(\bsx/s)\ind{\anchor(\bsx/s)=0}\ind{\norm{\bsx_0}>s}=
\ind{\min\{K(\bsx),\norm{\bsx_0}\}>s}\ind{ \bsx_{-\infty,-1}^*\leq s}
\;.
\end{align*}
Now, the class
$$\mcf=\{(-\infty,s)\times(s,+\infty):s \in \bbR \}$$
has the VC-index $3$. By \cite[Example C.4.14]{kulik:soulier:2020} the class $\mcg$ has VC-index at most 3.

Similarly, for $\anchor^{(2)}$ we have
\begin{align*}
\ind{\anchor^{(2)}(\bsx/s)=0}=
\ind{ \bsx_{1,\infty}^*\leq s,\norm{\bsx_0}>s}\;.
\end{align*}
Thus,
\begin{align*}
H(\bsx/s)\ind{\anchor(\bsx/s)=0}\ind{\norm{\bsx_0}>s}=
\ind{\min\{K(\bsx),\norm{\bsx_0}\}>s}\ind{ \bsx_{1,\infty}^*\leq s}
\;
\end{align*}
and again
the class $\mcg$ has VC-index at most 3.

In summary, for functionals $H$ given in \eqref{eq:class-1} and the anchoring maps $\anchor^{(0)}$, $\anchor^{(1)}$, $\anchor^{(2)}$ the class $\mcg$ has the VC-index at most 3 and hence the random entropy \Cref{hypo:entropy} is satisfied.

Now, assume that the map $s\to H_s$ is decreasing. This is the case of (again) the extremal index, the large deviation index and the ruin index. This is also the case of the stop-loss index and the cluster size distribution. If we choose $\anchor=\anchor^{(0)}$, since $\anchor^{(0)}$ is $0$-homogeneous, the maps $s\to H_s^{\anchor}$ is also decreasing. Thus, the VC-index of $\mcg$ is at most 2. The random entropy condition is satisfied.

\bibliographystyle{plain}

\newpage

\begin{table*}
	\centering
	\ra{1.5}
	\noindent\makebox[0.4\textwidth]{
		\begin{tabular}{@{}rrrrrcrrrrc@{}}\toprule \\
			& \multicolumn{4}{c}{\textcolor{blue}{$\rho=0.9$},$\;\;\;$\textcolor{red}{\text{Extremal Index}=$0.34$}} & \phantom{abc}& \multicolumn{4}{l}{\textcolor{blue}{$\rho=0.5$},$\;\:$ \textcolor{red}{\text{Extremal Index=} $0.94$}}\\
			\cmidrule{2-5} \cmidrule{7-10}
			($k$ $\,$\%)&\multicolumn{2}{c}{$k=5$} & \multicolumn{2}{c}{$k=10$} && \multicolumn{2}{c}{$k=5$}   & \multicolumn{2}{c}{$k=10$}  \\ \midrule
			\textcolor{purple}{$r_n=8\;$}\\
			Disjoint bl & \textbf{0.360} & (0.05)& \textbf{0.310 }& (0.03) && \textbf{0.680} & (0.05) & \textbf{0.570} &(0.03) \\
			Sliding bl &\textbf{ 0.348} & (0.04) & \textbf{0.308} & (0.03)&& \textbf{0.673}& (0.04)& \textbf{0.562}&(0.03)\\
			Runs $C^{(0)}$ &\textbf{ 0.240} & (0.05) & \textbf{0.190} & (0.03)&& \textbf{0.560}& (0.06)& \textbf{0.420}&(0.03)\\
			Runs $C^{(1)}$ &\textbf{ 0.220} & (0.05) & \textbf{0.170} & (0.03)&& \textbf{0.540}& (0.06)& \textbf{0.390}&(0.03)\\
			Runs $C^{(2)}$ &\textbf{ 0.220} & (0.05) & \textbf{0.170} & (0.03)&& \textbf{0.540}& (0.06)& \textbf{0.390}&(0.03)\\
			\textcolor{purple}{$r_n=9\;$}\\
			Disjoint bl & \textbf{0.340}& (0.05)& \textbf{0.290}& (0.03) && \textbf{0.655}& (0.05)& \textbf{0.540} & (0.03)\\
			Sliding bl & \textbf{0.340}& (0.04)& \textbf{0.309}& (0.03) && \textbf{0.672} & (0.04)& \textbf{0.539} & (0.03)\\
			Runs $C^{(0)}$ &\textbf{ 0.220} & (0.05) & \textbf{0.170} & (0.03)&& \textbf{0.520}& (0.05)& \textbf{0.370}&(0.03)\\
			Runs $C^{(1)}$ &\textbf{ 0.200} & (0.05) & \textbf{0.160} & (0.03)&& \textbf{0.500}& (0.05)& \textbf{0.340}&(0.03)\\
			Runs $C^{(2)}$ &\textbf{ 0.200} & (0.05) & \textbf{0.160} & (0.03)&& \textbf{0.500}& (0.05)&
			\textbf{0.340}&(0.03)\\
			\bottomrule
		\end{tabular}
	}
	\caption{The median and the variance (in brackets) of disjoint, sliding blocks and runs (with the anchoring maps $C^{(0)}, C^{(1)}$ and $C^{(2)}$) estimators for the extremal index. Data are simulated from AR(1) with $\alpha=4$, $\rho=0.5$ (thus, $\theta=0.94$), and $\rho=0.9$ (thus $\theta=0.34$). Block size $r_n=8,\;9$. The number of order statistics is $k=5\%$ and $10\%$ for a sample $n=1000$  based on $N = 1000$ Monte Carlo simulations.}
	\label{Summary table-AR-extremal}
\end{table*}

\begin{table*}
	\centering
	\ra{1.5}
	\noindent\makebox[0.4\textwidth]{
		\begin{tabular}{@{}rrrrrcrrrrc@{}}\toprule \\
			& \multicolumn{4}{c}{\textcolor{blue}{$\rho=0.9$},$\;\;\;$\textcolor{red}{\text{Stop-loss Index}=$0.085$}} & \phantom{abc}& \multicolumn{4}{l}{\textcolor{blue}{$\rho=0.5$},$\;\:$ \textcolor{red}{\text{Stop-loss Index=} $0.078$}}\\
			\cmidrule{2-5} \cmidrule{7-10}
			(k $\,$\%)&\multicolumn{2}{c}{$k=10$} & \multicolumn{2}{c}{$k=40$} && \multicolumn{2}{c}{$k=10$}   & \multicolumn{2}{c}{$k=40$}  \\ \midrule
			\textcolor{purple}{$r_n=8\;$}\\
			Disjoint bl & \textbf{0.0100}& (0.01)& \textbf{0.0175}& (0.007) && \textbf{0.0500}& (0.02)& \textbf{0.0675} & (0.009)\\
			Sliding bl & \textbf{0.0113}& (0.01)& \textbf{0.0181}& (0.007) && \textbf{0.0457} & (0.02)& \textbf{0.0678} & (0.008)\\
			Runs $C^{(0)}$ &\textbf{ 0.0200} & (0.01) & \textbf{0.0250} & (0.006)&& \textbf{0.0600}& (0.02)& \textcolor{red}{\textbf{0.0775}}&(0.007)\\
			\textcolor{purple}{$r_n=9\;$}\\
			Disjoint bl & \textbf{0.0100}& (0.01)& \textbf{0.0200}& (0.008) && \textbf{0.0500}&  (0.02)& \textbf{0.0700} & (0.009)\\
			Sliding bl & \textbf{0.0122}& (0.01)& \textbf{0.0197}&(0.007) && \textbf{0.0478} &  (0.02) &\textbf{0.0694} & (0.008)\\
			Runs $C^{(0)}$ &\textbf{ 0.0200} & (0.01) & \textbf{0.0250} & (0.005)&& \textbf{0.0600}& (0.02)& \textcolor{red}{\textbf{0.0750}}&(0.008)\\
			\bottomrule
		\end{tabular}
	}
	\caption{The median and the variance (in brackets) of disjoint, sliding blocks and runs ($C^{(0)}$) estimators for stop-loss index with $S=0.9$. Data are simulated from AR(1) with $\alpha=4$, $\rho=0.5,\;=0.9$. The block size is $r_n=8,\;9$. The number of order statistics is $k=10\%$ and $40\%$ for a sample $n=1000$  based on $N = 1000$ Monte Carlo simulations.}
	\label{Summary table stop loss}
\end{table*}

\begin{table*}
	\centering
	\ra{1.5}
	\noindent\makebox[0.4\textwidth]{
		\begin{tabular}{@{}rrrrrcrrrrc@{}}\toprule \\
			& \multicolumn{3}{c}{\textcolor{red}{\text{Extremal Index}=$0.612$}} & \phantom{abc}&\\
			\cmidrule{2-5} \cmidrule{7-10}
			($k$ $\,$\%)&\multicolumn{2}{c}{$k=5$} & \multicolumn{2}{c}{$k=10$} &&\\ \midrule
			\textcolor{purple}{$r_n=8\;$}\\
			Disjoint bl & \textbf{0.660}& (0.06)& \textbf{0.600}& (0.04) &&\\
			Sliding bl & \textbf{0.648}& (0.06)& \textbf{0.593}& (0.03) && \\
			Runs $C^{(0)}$ &\textbf{ 0.520} & (0.06) & \textbf{0.410} & (0.03)&& \\
			Runs $C^{(1)}$ & \textbf{0.500}& (0.06)& \textbf{0.380}&(0.03) &&\\
			Runs $C^{(2)}$ & \textbf{0.500}& (0.06)& \textbf{0.380}&(0.03) &&\\			
			\textcolor{purple}{$r_n=9\;$}\\
			Disjoint bl & \textbf{0.640}& (0.06)& \textbf{0.570}& (0.04) && \\
			Sliding bl &\textbf{ 0.630} & (0.06) & \textbf{0.569} & (0.03)&& \\
			Runs $C^{(0)}$ & \textbf{0.500}& (0.06)& \textbf{0.380}&(0.03) &&\\
			Runs $C^{(1)}$ & \textbf{0.480}& (0.06)& \textbf{0.340}&(0.03) &&\\
			Runs $C^{(2)}$ & \textbf{0.480}& (0.06)& \textbf{0.340}&(0.03) &&\\						
			\bottomrule
			
		\end{tabular}
	}
	\caption{The median and the variance (in brackets) of disjoint, sliding blocks and runs ($C^{(0)},\;C^{(1)}$ and $C^{(2)}$)  estimators for the extremal index in ARCH(1) model with $\lambda=0.9$. The block size is $r_n=8,\;9$. The number of order statistics is $k=5\%$ and $10\%$ for a sample $n=1000$  based on $N = 1000$ Monte Carlo simulations.}
	\label{Summary table-ARCH-extremal}
\end{table*}

\begin{figure}[!ht]
	\centering
	\includegraphics[width=9cm,height=5cm]{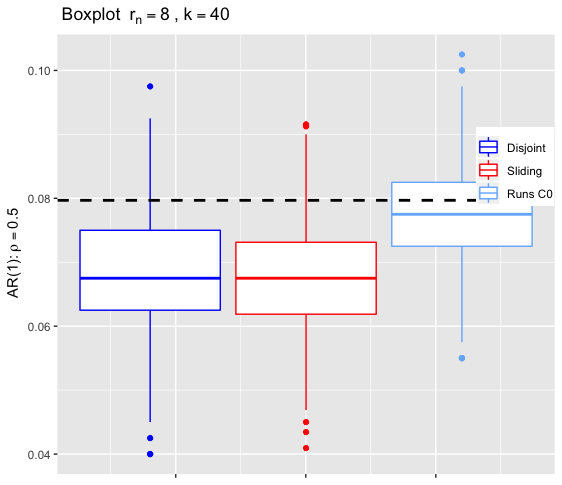}
	\includegraphics[width=9cm,height=5cm]{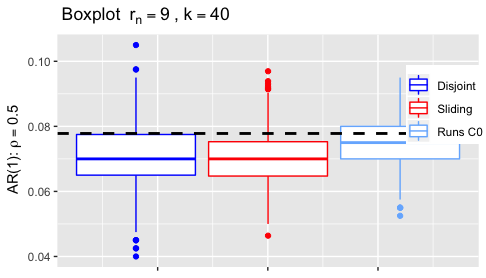}
	\caption{Monte Carlo simulations of runs estimator ($C^{(0)}$) for stop-loss index with $S=0.9$. Data are simulated from AR(1) with $\rho=0.5$ and $r_n=8$  (left panel), $\rho=0.9$ and $r_n=9$ (right panel), $\theta=0.078$ $\alpha=4$ and the number of order statistics $k=40$. Dotted lines indicated the true value of the cluster index.}
	\label{plot:box plots}
\end{figure}

\end{document}